\documentclass[11pt]{amsart}
\usepackage{amssymb}
\usepackage[T1]{fontenc}
\usepackage[utf8]{inputenc}
\usepackage{lmodern}
\usepackage[margin=1in]{geometry}
\usepackage{amsmath,amssymb,amsthm,mathtools}
\usepackage{enumitem}
\usepackage[colorlinks=true,linkcolor=blue,citecolor=blue,urlcolor=blue]{hyperref}
\usepackage{tikz-cd}
\usepackage{tikz}
\usetikzlibrary{arrows.meta}
\usepackage{amsmath,amssymb,amsthm}
\theoremstyle{plain}
\newtheorem{rrhutheorem}{Theorem}[section]
\newtheorem{rrhucorollary}[rrhutheorem]{Corollary}
\newtheorem{rrhuquestion}[rrhutheorem]{Question}

\newcommand{\htop}{h_{\mathrm{top}}}

\newcommand{\cT}{\mathcal T}

\usepackage{tcolorbox}

\newtheorem{theorem}{Theorem}[section]
\newtheorem{proposition}[theorem]{Proposition}
\newtheorem{lemma}[theorem]{Lemma}
\newtheorem{corollary}[theorem]{Corollary}
\newtheorem{claim}{Claim}[section]

\theoremstyle{definition}
\newtheorem{definition}[theorem]{Definition}
\newtheorem{question}[theorem]{Question}

\theoremstyle{remark}
\newtheorem{remark}[theorem]{Remark}

\numberwithin{equation}{section}
\usepackage{hyperref}

\usepackage{cleveref}

\providecommand{\cT}{\mathcal{T}}
\providecommand{\cZ}{\mathcal{Z}}
\providecommand{\cF}{\mathcal{F}}

\author{Pablo D. Carrasco}
\address[P.D. Carrasco]{ICEx-UFMG, Avda. Presidente Antonio Carlos 6627, Belo Horizonte-MG, BR31270-90}
\email{pdcarrasco@ufmg.br}
\thanks{P.D. Carrasco was partially supported by CNPq (Produtividade em Pesquisa) 305197/2025-8, Projeto Universal CNPq 401737/2025-0, FAPEMIG APQ-02338-24, and CAPES through the MATH-AMSUD program.}

\author{Federico Rodriguez-Hertz}
\address[F. Rodriguez-Hertz]{Penn State, 227 McAllister Building, University Park, State College, PA16802}
\email{hertz@math.psu.edu}
 \thanks{F. Rodriguez-Hertz was partially supported by NSF grant DMS-2453688.}

\author{Jana Rodriguez-Hertz}
\author{Raúl Ures}
\address[J. Rodriguez-Hertz and R. Ures]
{ Department of Mathematics and Shenzhen International Center for Mathematics - Southern University of Science and Technology,
1088 Xueyuan Avenue, Shenzhen 518055, P.R. China}
\email[J. Rodriguez-Hertz]{rhertz@sustech.edu.cn}
\email[R. Ures]{ures@sustech.edu.cn}
\title{Entropy and periodic points in pseudo-Anosov isotopy classes}
\keywords{Pseudo-Anosov homeomorphism, Nielsen classes, entropy of maximal entropy}
\begin{document}

\begin{abstract}
Let $f$ be a pseudo-Anosov homeomorphism of a closed orientable surface of genus at least two. We prove that, for every $C^\infty$ diffeomorphism $g$ isotopic to $f$, the exponential growth rate of the number of nonempty Nielsen classes of its iterates equals $h_{\mathrm{top}}(g)$ if and only if
$h_{\mathrm{top}}(g)=h_{\mathrm{top}}(f)$. Thus Nielsen growth captures the full entropy precisely for representatives minimizing entropy in their isotopy class. For $C^{1+\alpha}$ diffeomorphisms, $\alpha>0$, we identify this growth rate with the entropy of the global orbit quotient in Fathi's hyperbolic extension and establish Bowen's periodic-point formula for the quotient. We also prove that quotient entropy is upper
semicontinuous in the $C^0$ topology throughout the isotopy class and continuous at every $C^{1+\alpha}$ diffeomorphism, even under perturbations by arbitrary homeomorphisms in the class.
\end{abstract}
\maketitle

\section{Introduction}

Let $M$ be a closed connected orientable surface of genus at least two. The Nielsen--Thurston classification \cite{Thurston1988} describes its mapping classes in terms of periodic, reducible, and pseudo-Anosov representatives. Periodic representatives have finite order. Reducible
representatives preserve a nonempty finite collection of disjoint essential simple closed curves; after cutting along a suitable such collection and passing to an iterate, the induced mapping classes on the complementary components are periodic or pseudo-Anosov. The pseudo-Anosov pieces
are therefore the source of positive entropy in the canonical decomposition. Their importance for dynamics lies in the complexity they impose on every representative of the isotopy class.

We fix a pseudo-Anosov homeomorphism $f\colon M\to M$. It preserves two transverse singular measured foliations, multiplying their transverse measures by $\lambda$ and $\lambda^{-1}$, where $\lambda>1$ is its stretch factor. This expansion and contraction determine its entropy,
which is minimal in its isotopy class:
\[
  h_{\mathrm{top}}(f)=\log\lambda
  =\inf\{\htop(g): g\text{ isotopic to }f\}.
\]
Thus the isotopy class already requires positive entropy, while its representatives may have additional dynamical complexity.

The complexity of map is typically related to the presence of periodic points. In the case of surfaces, the powerful Nielsen's fixed point theory \cite{Nielsen1927} provides a way to group periodic points according to some topological criterion. Let $g\simeq f$ and $n\geq1$ where, from now on, $g\simeq f$ means that $g$ is isotopic to $f$.

\begin{definition}\label{def:NielsenEquivalent}
Two points $p,q\in\operatorname{Fix}(g^n)$ are \emph{Nielsen equivalent} if there is a path $c$ from $p$ to $q$ such that $c$ and
$g^n\circ c$ are homotopic while their endpoints remain fixed.
\end{definition}

The resulting equivalence classes are the \emph{Nielsen classes of $g^n$}. We denote
\[
  C_n(g)=\#\{\text{nonempty Nielsen classes of }g^n\},
\]
and define
\begin{align}\label{eq:Nielsengrowth}
h_N(g)=\limsup_{n\to\infty}
  \frac1n\log\max\{1,C_n(g)\}.
\end{align}

We count every nonempty class. Our purpose is to determine how much of the dynamical complexity of $g$ is detected by the growth of these classes, and when this growth captures its full topological entropy.

The basic bounds are
\[
  h_{\mathrm{top}}(f)=h_N(f)
  \leq h_N(g)\leq h_{\mathrm{top}}(g).
\]
Handel's results \cite[Corollary~3.6]{Handel1988} characterize the minimum: $h_N(g)=h_{\mathrm{top}}(f)$
precisely when $g$ admits a semiconjugacy onto $f$ homotopic to the identity. To compare Nielsen growth with the entropy of other representatives, we use Fathi's hyperbolic extension \cite{Fathi1990} and the global orbit quotients constructed in \cite{HHU2004}. These quotients place the dynamics of every representative inside a common ambient system.

Let $\widetilde M$ be the universal cover of $M$, with deck group $\Gamma$, and fix a lift $F$ of $f$. Collapsing each leaf of the lifted stable or unstable foliation gives a real tree $\mathcal T^s$ or $\mathcal T^u$, respectively, whose metric is induced
by the transverse measure. Write $\widehat{\mathcal T}^s$ and $\widehat{\mathcal T}^u$
for their metric completions, and let
\begin{align*}
&\mathcal T=\widehat{\mathcal T}^s\times\widehat{\mathcal T}^u\\
&Z=\mathcal T/\Gamma,
\end{align*}
where $\Gamma$ acts diagonally. The lift $F$ induces a product homeomorphism $\widetilde\Phi$ of $\mathcal T$, expanding distances in $\widehat{\mathcal T}^s$ by $\lambda$ and contracting distances in $\widehat{\mathcal T}^u$ by $\lambda^{-1}$. The superscript records the foliation being collapsed:
the stable-leaf tree is the expanding coordinate. The product map descends to a homeomorphism $\Phi\colon Z\to Z$. The pairs of intersecting stable
and unstable leaves determine an embedded copy $L(f)$ of $M$ in $Z$, and $\Phi|_{L(f)}$ is conjugate to $f$.

For $g\simeq f$, choose the lift $G$ obtained by lifting an isotopy from $f$ to $g$, starting at $F$.

\begin{definition}\label{def:globalequivalent}
Two points $x,y\in M$ are globally orbit equivalent if suitable lifts satisfy
\[
  x\sim_g y   \quad\Longleftrightarrow\quad  \sup_{n\in\mathbb Z} D(G^n\widetilde x,G^n\widetilde y)<\infty,
\]
where $D$ is the lifted metric induced by the measured foliations.  
\end{definition}
Thus this relation identifies points whose lifted orbits remain at bounded distance in both time directions. Fathi's construction gives a continuous surjection $H_g\colon M\to L(g)$ onto a compact $\Phi$-invariant subset of $Z$ such that
\[
  H_g\circ g=\Phi_g\circ H_g,
  \qquad
  \Phi_g=\Phi|_{L(g)},
  \qquad
  L(f)\subset L(g).
\]
The last three authors of this article identified the fibers of $H_g$ precisely \cite{HHU2004}:
\[
  H_g(x)=H_g(y)   \quad\Longleftrightarrow\quad  x\sim_g y.
\]
It follows that $L(g)$ realizes the quotient $M/{\sim_g}$, and $\Phi_g$ realizes its induced dynamics. The ambient map $\Phi$ is fixed throughout the isotopy class; the representative $g$ determines the invariant set $L(g)$ on which it is restricted.

The inclusion $L(f)\subset L(g)$ expresses the persistence of the pseudo-Anosov dynamics under global shadowing. Handel \cite{Handel1985} proved that every $g\simeq f$ has a compact invariant subset whose dynamics factors onto $f$; in this construction, that subset is
\[
  J_g=H_g^{-1}(L(f)).
\]
The complement $L(g)\setminus L(f)$ records the global orbit classes that do not correspond to orbits of $f$.

For periodic points, the relation with Nielsen classes is particularly direct. If $p,q\in\operatorname{Fix}(g^n)$,
then $H_g(p)=H_g(q)$ if and only if $p$ and $q$ are Nielsen equivalent \cite[Lemma~2.8]{LU2003}. Thus each nonempty Nielsen class of $g^n$ determines
a point of $\operatorname{Fix}(\Phi_g^n)$, and distinct classes determine distinct points. Our first theorem identifies the exponential growth rates of these two counts with the entropy of the quotient.

\begin{theorem}[Entropy formula]
\label{thm:entropy.formula}
Let $g$ be a $C^{1+\alpha}$ diffeomorphism, $\alpha>0$,
or a $C^1$ Axiom-A diffeomorphism, in the isotopy class
of $f$. Then
\[
  h_N(g)=h_{\mathrm{top}}(\Phi_g) =\limsup_{n\to\infty} \frac1n\log\max\{1,\#\operatorname{Fix}(\Phi_g^n)\}.
\]
\end{theorem}

This answers \cite[Question~4.2]{HHU2004} under the stated regularity hypotheses. In particular, the quotient satisfies Bowen's periodic-point entropy
formula. Although $\Phi_g$ is expansive, this formula requires an argument beyond the hyperbolicity of the ambient extension: the sets $L(g)$ are not locally maximal, and the ambient shadowing property does not pass to their restrictions.

Our second theorem compares the entropy of the quotient with that of the surface dynamics. The geometry of $L(g)$ is essential here. These spaces need not be surfaces: their local stable and unstable components are dendrites, that is, locally connected continua containing no simple closed curve. Every periodic point outside $L(f)$ is an endpoint of both components \cite{HHU2004}. We use this endpoint geometry together with the recent method developed by the authors \cite{CHHU2026}, that relates entropy and semiconjugacy on surfaces. We emphasize that several additional difficulties appear while dealing with global orbit quotients, due to lack of any differentiable structure.


\begin{theorem}[Strict entropy gap]\label{thm:strict-entropy-gap}
If $g$ is a $C^\infty$ diffeomorphism isotopic to $f$,
then
\[
  h_{\mathrm{top}}(f)<h_{\mathrm{top}}(\Phi_g) \quad\Longrightarrow\quad
  h_{\mathrm{top}}(\Phi_g)<h_{\mathrm{top}}(g).
\]
\end{theorem}

Thus, whenever the quotient entropy exceeds the pseudo-Anosov minimum, it is strictly smaller than
the entropy of the representative. Together with \Cref{thm:entropy.formula} and the basic entropy bounds, this shows that Nielsen growth captures the full entropy of a smooth representative precisely when that representative minimizes entropy in its isotopy class. We record this consequence as
\Cref{cor:entropy.minimization} in Section~\ref{sec:entropy.formula}.

Our third theorem concerns the dependence of quotient entropy on the representative.

\begin{theorem}[Continuity of quotient entropy]\label{thm:quotient-entropy-continuity}
On the space of homeomorphisms isotopic to $f$, the map
\[
  g\longmapsto h_{\mathrm{top}}(\Phi_g)
\]
is upper semicontinuous in the $C^0$ topology. Moreover, let $g\simeq f$ be a $C^{1+\alpha}$ diffeomorphism, with $\alpha>0$. For every sequence of homeomorphisms $g_j\simeq f$ converging uniformly to $g$, one has
\[
  h_{\mathrm{top}}(\Phi_{g_j}) \longrightarrow h_{\mathrm{top}}(\Phi_g).
\]
No differentiability is required of the approximating homeomorphisms $g_j$.
\end{theorem}

Quotient entropy therefore refines the pseudo-Anosov entropy bound while remaining stable, at these representatives, under perturbations that impose
no control on derivatives. This answers the continuity part of \cite[Question~4.1]{HHU2004} at $C^{1+\alpha}$ diffeomorphisms. Together with \Cref{thm:entropy.formula}, this implies that if $g_j$ and $g$ are $C^{1+\alpha}$ diffeomorphisms isotopic to $f$ and $g_j\to g$ uniformly, then
$h_N(g_j)\to h_N(g)$.

The invariant foliations of $f$ also connect these entropy questions with Teichm\"uller theory. They are the horizontal
and vertical foliations of a quadratic differential and determine a geodesic in Teichm\"uller space, the space of
marked conformal structures on $M$. The mapping class of $f$ translates along this geodesic by $\log\lambda$
\cite{Bers1978}. The minimum entropy in the isotopy class thus has a geometric interpretation: it is the translation
length of the class on Teichm\"uller space. The projective classes of the two foliations are its two fixed points on
Thurston's boundary.

These same foliations give the dual trees used in this paper. The translation lengths of the deck transformations on each
tree record intersection numbers of closed curves with the corresponding measured foliation. Such tree actions belong
to the geometric picture developed by Morgan and Shalen \cite{MorganShalen1984}, in which real trees describe degenerations of hyperbolic structures. Fathi's construction \cite{Fathi1990} brings the two trees together in the hyperbolic system $(Z,\Phi)$, and the sets $L(g)$ of
\cite{HHU2004} realize the global orbit quotients of representatives throughout the isotopy class inside this
common space. The geometry supplied by $f$ therefore also allows us to study the additional complexity of $g$.
Our results compare the entropy retained by this quotient with both the geometric minimum $\log\lambda$ and the
full entropy of the representative.

We note that in a recent preprint of Garc{\'i}a-Sassi and Tal \cite{GarciaSassiTal2026}, the authors consider a related quotient construction for maps in the isotopy class of a pseudo-Anosov, by collapsing maximal continua with bounded infinite orbit in the universal cover. The resulting quotient map sits above Fathi's construction, and is CW-expansive. It seems very likely that some methods in this paper can be applied in that situation; investigating such extent remains as an interesting open problem.

The proofs proceed as follows. Section~\ref{sec:hhu-preliminaries} recalls the dual trees, Fathi's extension, and the realization
of the global orbit quotient. Section~\ref{sec:hyperbolic-preliminaries} collects the preliminaries on hyperbolic measures, homoclinic
classes, conditional measures, and horseshoes.

Section~\ref{sec:entropy.formula} proves \Cref{thm:entropy.formula,thm:strict-entropy-gap} and derives Corollary \ref{cor:entropy.minimization}.
For the entropy formula, the correspondence between Nielsen classes and images of periodic points, together with expansivity, gives one inequality. For the reverse inequality, we lift positive-entropy quotient measures to hyperbolic measures on the surface. The Brin--Katok formula \cite{BK1983}
supplies many return segments whose images in the quotient are separated; Katok's closing lemma \cite{Katok1980} turns these segments into periodic
points belonging to distinct Nielsen classes.

The proof of the strict gap develops the geometric argument of \cite{CHHU2026}. Under entropy equality above $h_{\mathrm{top}}(f)$, a maximal-entropy measure of the quotient lifts to a hyperbolic maximal-entropy measure with zero relative entropy over $H_g$.
The conditional-measure argument recalled in Section~\ref{subsec:prelim-conditionals} then gives injectivity on full conditional-measure subsets of stable and unstable manifolds. Homoclinic periodic approximation provides a saddle projecting outside $L(f)$, where the endpoint
geometry forces a fold in the tree image of its unstable manifold. The HHTU Sard argument \cite{HHTU2011SRB}, in its adaptation to
maximal-entropy measures \cite{BCS2022}, supplies transverse stable crossings of the folded arc. Ben Ovadia's absolute continuity result
\cite{BO2023} allows these crossings to be chosen in the set of injectivity, giving a contradiction. The $C^\infty$ hypothesis enters through the Sard
argument.

Section~\ref{sec:quotient-entropy-continuity} proves \Cref{thm:quotient-entropy-continuity}. Upper semicontinuity follows from the Hausdorff
continuity of $L(g)$ inside the common expansive ambient system. For lower semicontinuity, we incorporate returns separated in the quotient into
the Katok--Mendoza construction, obtaining locally maximal horseshoes on which $H_g$ is injective. Nitecki's $C^0$ semistability theorem
\cite{Nitecki1971}, together with expansivity, ensures that their images in the quotient persist under nearby homeomorphisms.

Section~\ref{sec:questions} presents the remaining questions.

For convenience of the reader we reproduce the  unpublished article \cite{HHU2004} in the Appendix.

\section{Dual trees and the universal hyperbolic extension}
\label{sec:hhu-preliminaries}

Let $M$ be a closed connected orientable surface of genus at least two, and let $f\colon M\to M$ be a pseudo-Anosov homeomorphism with dilatation
$\lambda>1$. We denote by $\cF^s, \cF^u$ its stable and unstable measured foliations, respectively. Fix a universal covering $p\colon\widetilde M\to M$, with deck group $\Gamma$, and a lift $F\colon\widetilde M\to\widetilde M$ of $f$.

We recall Fathi's hyperbolic extension~\cite{Fathi1990} and the realization of the orbit quotients in~\cite{HHU2004}.

\subsection{The dual trees}
\label{subsec:hhu-dual-trees}
For $\sigma\in\{s,u\}$, let
\[
  \cT^\sigma=\widetilde M/\widetilde{\cF}^{\sigma},
  \qquad q_\sigma\colon\widetilde M\to\cT^\sigma,
\]
be the space of leaves of the lifted foliation. A singular leaf includes the singularity and all its separatrices. The transverse measure defines
a metric
\[
  d_\sigma(\ell,\ell')
  =\inf\{\text{transverse measure of }\alpha:
             \alpha\text{ joins }\ell\text{ to }\ell'\}.
\]
These leaf spaces are $\mathbb R$-trees, as are their metric completions $\widehat{\cT}^{\sigma}$; see~\cite[\S2]{MorganShalen1984}
and~\cite[\S1]{Fathi1990}. The trees $\cT^\sigma$ are not complete, nor they are locally compact. Every point of $\widehat{\cT}^{\sigma}\setminus\cT^\sigma$ is an endpoint of the completed tree~\cite[Lemma~1.1]{Fathi1990}.

The deck transformations act isometrically on both trees and their completions. The maps induced by $F$ extend to homeomorphisms $F_\sigma\colon\widehat{\cT}^{\sigma}\to\widehat{\cT}^{\sigma}$ satisfying 
\begin{equation}
\label{eq:hhu-tree-scaling}
\begin{aligned}
 d_s(F_s a,F_s b)&=\lambda\,d_s(a,b),\\
 d_u(F_u a,F_u b)&=\lambda^{-1}d_u(a,b).
\end{aligned}
\end{equation}
Here the superscript on $\cT^\sigma$ records the foliation being collapsed, as in~\cite{HHU2004}: the stable-leaf tree carries the expanding coordinate,
and the unstable-leaf tree carries the contracting coordinate. We refer to these spaces as the \emph{dual trees of the invariant foliations}, and to the construction below as \emph{Fathi's hyperbolic extension}.

\subsection{The product and the ambient map}
\label{subsec:hhu-ambient-map}

Let
\[
 \cT=\widehat{\cT}^{s}\times\widehat{\cT}^{u}
\]
equipped with the metric
\[
 \widetilde d(z,w)=\max\{d_s(z_s,w_s),d_u(z_u,w_u)\}. 
\]

Then $\cT$ is the complete product of the two trees, denoted $\widetilde{\mathcal Z}$ in~\cite{HHU2004}. Define $\widetilde\Phi\colon\cT\to\cT$,
\[
\widetilde\Phi(z_s,z_u)=(F_s z_s,F_u z_u).
\]
The metric splitting is global: for all $n\in\mathbb Z$,
\begin{equation}\label{eq:hhu-global-splitting}
 \widetilde d(\widetilde\Phi^n z,\widetilde\Phi^n w)
 =\max\{\lambda^n d_s(z_s,w_s),
             \lambda^{-n}d_u(z_u,w_u)\}.
\end{equation}
Hence we get (cf. \cite[Theorem~1.1]{HHU2004}).

\begin{lemma}\label{lem:propertiesofwidetildephi}
  The following properties hold:
  \begin{enumerate}
    \item for every $z\in \cT$, 
    \[
    W^s(z;\widetilde\Phi)=\{z_s\}\times\widehat{\cT}^{u},
    \qquad
    W^u(z;\widetilde\Phi)=\widehat{\cT}^{s}\times\{z_u\},
    \]
   \item $\widetilde\Phi$ is expansive.

   \item $\widetilde\Phi$ has the shadowing property.
  \end{enumerate}
\end{lemma}

\begin{proof}
  The first part is direct from \eqref{eq:hhu-global-splitting}. It implies, in particular,
\[
 W^s(z;\widetilde\Phi)\cap W^u(w;\widetilde\Phi)=\{(z_s,w_u)\}.
\]
so any two points whose $\widetilde\Phi$-iterates remain at bounded distance for all integer times, necessary coincide. So, $\widetilde\Phi$ is expansive. In fact, \eqref{eq:hhu-global-splitting} implies the stronger following condition: given $\epsilon>0$, if $N\in\mathbb N$ then 

\[
\sup_{n\in [-N,N]}\widetilde d\left(\widetilde\Phi^nz_1,\widetilde\Phi^nz_2\right)<\epsilon\Rightarrow \widetilde d(z_1,z_2)<\frac{\epsilon}{\lambda^N}.
\]
Fix $\epsilon>0$. Bowen's proof of shadowing \cite{EquSta} gives $\delta>0$ so that, if $(z_n)_{n=-\infty}^{n=\infty}$ is a $\delta$-pseudo orbit of $\widetilde\Phi$,  for every $M\in\mathbb N$ there exists $z^{(M)}\in \cT$ whose $[-M,M]$ orbit $\epsilon$-shadows $(z_n)_{n=-M}^{n=M}$. It follows that for $M'>M$, $\widetilde d(z^{(M)},z^{(M')})\leq \frac{2\epsilon}{\lambda^M}$, hence $(z^{(M)})_{M\geq 1}$ is Cauchy. Since $\cT$ is complete, it has a limit $z\in\cT$. Clearly, the orbit of $z$ shadows $(z_n)_{n=-\infty}^{n=\infty}$ at distance at most $2\epsilon$.
\end{proof}

Let $\varphi\in\operatorname{Aut}(\Gamma)$ be determined by $F\gamma=\varphi(\gamma)F$. The diagonal action of $\Gamma$ on $\cT$ is isometric and properly discontinuous~\cite[Lemma~2.1 and Corollary~2.2]{Fathi1990}. We consider the orbit space
\[
 \cZ=\cT/\Gamma,
 \qquad P\colon\cT\to\cZ
\]
equipped with the metric
\[
 d(Pz,Pw)=\inf_{\gamma\in\Gamma}\widetilde d(z,\gamma w).
\]
$P$ is the universal covering of $\cZ$, and $\cZ$ is complete. Since $\widetilde\Phi\gamma=\varphi(\gamma)\widetilde\Phi$, the product map
induces a homeomorphism
\[
 \Phi\colon\cZ\to\cZ, \qquad \Phi(Pz)=P(\widetilde\Phi z).
\]

\begin{lemma}\label{lem:propertiesofphi}
$\Phi$ has hyperbolic local product coordinates. Moreover, 
\begin{enumerate}
  \item it is expansive,
  \item it has the shadowing property. 
\end{enumerate}
\end{lemma}
The hyperbolic local product coordinates are induced by the global metric splitting upstairs. In particular, $(\Phi,\cZ)$ is a Smale Space in the sense of Ruelle.

There is a natural equivariant embedding $\iota\colon\widetilde M\to\cT$,
\[
 \iota(x)=(q_s(x),q_u(x))
\]
whose image is
\[
 \widetilde L(f)= \{(\ell^s,\ell^u)\in\cT^s\times\cT^u:\ell^s\cap\ell^u\ne\varnothing\}.
\]
It satisfies  
\[
\begin{tikzcd}
\widetilde M \arrow[r, "F"] \arrow[d, "\iota"'] &
\widetilde M \arrow[d, "\iota"] \\
\mathcal T \arrow[r, "\widetilde\Phi"'] &
\mathcal T
\end{tikzcd}
\]
and induces an embedding $\iota_f\colon M\to\cZ$. Write
\[
L(f)=\iota_f(M). 
\]
Then $\Phi|_{L(f)}$ is conjugate to $f$ through $\iota_f$. Every nontrivial closed connected $\Phi$-invariant subset of $\cZ$ contains
$L(f)$~\cite[Proposition~2.5]{Fathi1990}.

\subsection{Global shadowing and the orbit quotient}
\label{subsec:hhu-orbit-quotient}

Let $g\colon M\to M$ be a homeomorphism isotopic to $f$. Choose the lift $G$ obtained by lifting an isotopy from $f$ to $g$ starting at $F$;
thus $G\gamma=\varphi(\gamma)G$. We use on $\widetilde M$ the metric $D(x,y)=\widetilde d(\iota(x),\iota(y))$.

Fathi's construction gives a continuous $\Gamma$-equivariant map $\widetilde H_g\colon \widetilde M\to\mathcal T$ fitting into the commutative diagram
\begin{equation}\label{eq:hhu-lifted-semiconjugacy}
\begin{tikzcd}
\widetilde M \arrow[r, "G"] \arrow[d, "\widetilde H_g"'] &
\widetilde M \arrow[d, "\widetilde H_g"] \\
\mathcal T \arrow[r, "\widetilde\Phi"'] &
\mathcal T,
\end{tikzcd}
\end{equation}
and satisfying
\begin{align}\label{eq:hhu-lifted-semiconjugacy-metric}
\sup_{x\in\widetilde M}\widetilde d\bigl(\widetilde H_g(x),\iota(x)\bigr)<\infty.
\end{align}

Its two coordinates may be written explicitly as
\begin{equation}
\label{eq:hhu-coordinate-limits}
 \widetilde H_g(x)=
 \left(
   \lim_{n\to\infty}F_s^{-n}q_s(G^n x),\,
   \lim_{n\to\infty}F_u^{n}q_u(G^{-n}x)
 \right).
\end{equation}
The limits exist in the completed trees: consecutive terms are at distances bounded by a constant times $\lambda^{-n}$, because $F$ and
$G$, and their inverses, are at bounded distance. Equivalently, $\widetilde H_g(x)$ is the unique point $z\in\cT$ satisfying
\begin{equation}
\label{eq:hhu-shadowing-characterization}
 \sup_{n\in\mathbb Z}
   \widetilde d(\widetilde\Phi^n z,\iota(G^n x))<\infty.
\end{equation}

Descending to the quotient defines $H_g\colon M\to\cZ$ by the diagram
\[
\begin{tikzcd}
\widetilde M \arrow[d, "p"'] \arrow[r, "\widetilde H_g"] &
\mathcal T \arrow[d, "P"] \\
M \arrow[r, "H_g"'] &
\cZ,
\end{tikzcd}
\]
Denote
\begin{align}
\notag &L(g)=H_g(M)\\
&\Phi_g=\Phi|_{L(g)}\colon L(g)\to L(g).
\label{eq:hhu-Lg-Phig}
\end{align}

Then
\[
\begin{tikzcd}
M \arrow[r, "g"] \arrow[d, "H_g"'] &
M \arrow[d, "H_g"] \\
L(g) \arrow[r, "\Phi_g"'] &
L(g)
\end{tikzcd}
\]

Thus $\Phi$ is the ambient homeomorphism determined by $f$, while $\Phi_g$ is its restriction to the compact invariant set determined by
$g$. In particular, changing $g$ changes $L(g)$ inside the same ambient system $(\cZ,\Phi)$.

To describe precisely which orbits are identified, recall Definition~\ref{def:NielsenEquivalent}: $x\sim_g y$ if there are lifts $\widetilde x,\widetilde y\in\widetilde M$ such that
\begin{equation}
\label{eq:hhu-global-orbit-equivalence}
 \sup_{n\in\mathbb Z}D(G^n\widetilde x,G^n\widetilde y)<\infty.
\end{equation}

We denote
\begin{align*}
 &M_L(g)=M/{\sim_g}\\
 &g_L([x])=[g(x)]
\end{align*}
Then $g_L \colon M_L(g)\to M_L(g)$ is a homeomorphism. Theorem~1.1 of \cite{HHU2004} proves that
\[
 x\sim_g y\quad\Longleftrightarrow\quad H_g(x)=H_g(y):
\]
It follows that $[x]\mapsto H_g(x)$ is a homeomorphism $M_L(g)\to L(g)$, conjugating $g_L$ to $\Phi_g$.

For comparison with the original pseudo-Anosov dynamics, the notation
\[
(f,x)\sim(g,y) 
\]
means that suitable lifts satisfy
\[
 \sup_{n\in\mathbb Z}D(F^n\widetilde x,G^n\widetilde y)<\infty.
\]
Equivalently, $H_g(y)=\iota_f(x)$. Every $f$-orbit is globally shadowed by a $g$-orbit, so $L(f)\subset L(g)$. 

With $J_g=H_g^{-1}(L(f))$, the map $\iota_f^{-1}\circ H_g|_{J_g}$ semiconjugates $g|_{J_g}$ onto
$f$~\cite[Theorem~2.6]{Fathi1990}.

The assignment $g\mapsto L(g)$ is continuous from the $C^0$ topology to the Hausdorff topology on compact subsets of $\cZ$. If $h\colon M\to M$
is homotopic to the identity and $h\circ g=g'\circ h$, then $L(g)=L(g')$~\cite[Theorem~1.2 and Corollary~1.3]{HHU2004}.

\subsection{Local structure and periodic points}
\label{subsec:hhu-local-structure}

The space $L(g)$ is a compact locally connected continuum of topological dimension two, and $\Phi_g \colon L(g)\to L(g)$ is expansive. For $z\in L(g)$, write
\[
\begin{aligned}
 W^s_\varepsilon(z;\Phi_g)
 &=\{w\in L(g):d(\Phi_g^n z,\Phi_g^n w)\le\varepsilon
                         \text{ for every }n\ge0\},\\
 W^u_\varepsilon(z;\Phi_g)
 &=W^s_\varepsilon(z;\Phi_g^{-1}),
\end{aligned}
\]
and let $CW^\sigma_\varepsilon(z;\Phi_g)$ denote the connected component containing $z$. By~\cite[Theorem~1.4]{HHU2004}, there are
$\varepsilon,c>0$ such that these components are nontrivial dendrites and
\[
 \operatorname{diam}CW^\sigma_\varepsilon(z;\Phi_g)\ge c
 \qquad(z\in L(g),\ \sigma\in\{s,u\}).
\]
Every periodic point of $\Phi_g$ outside $L(f)$ is an endpoint of both its stable and its unstable dendrite. The ambient shadowing property does
not extend to these restrictions: $L(g)$ is not an isolated invariant set of $\Phi$, and $\Phi_g$ has no shadowing property, nor is topologically stable~\cite[\S2 and Corollary~1.3]{HHU2004}.

Finally, recall Definition~\ref{def:NielsenEquivalent} for $p,q\in\operatorname{Fix}(g^n)$: it is equivalent to say that one lift of $g^n$ fixes suitable lifts of both points.

In \cite[Lemma~2.8]{LU2003} it is proven that
\[
 p\text{ and }q\text{ are Nielsen equivalent for }g^n  \quad\Longleftrightarrow\quad H_g(p)=H_g(q).
\]
Thus, each nonempty Nielsen class of $g^n$ determines one point of $\operatorname{Fix}(\Phi_g^n)$, and different classes determine different
points. This concerns all nonempty classes, including those with zero fixed-point index.

\section{Hyperbolic measures and homoclinic classes}
\label{sec:hyperbolic-preliminaries}

We recall some facts from nonuniform hyperbolicity that will be used to study the invariant measures of $\Phi_g$ through their lifts to $M$.
Throughout this section, $g$ is a $C^{1+\alpha}$ surface diffeomorphism, $\alpha>0$, unless greater regularity is explicitly assumed. We fix a
smooth Riemannian metric on $M$, with distance $d_M$ and tangent norm $\|\cdot\|$. Invariant measures are understood to be Borel probability
measures. We write $h_\mu(g)$ for the metric entropy of an invariant measure $\mu$; it is a \emph{measure of maximal entropy} if
$h_\mu(g)=h_{\mathrm{top}}(g)$.
\subsection{Lyapunov exponents and Pesin manifolds}
\label{subsec:prelim-pesin}

The Lyapunov exponents describe the exponential growth of tangent vectors under the derivative of the iterates of $g$. We refer the reader to \cite{LyaViana} for the proof of the following facts.

The classical Furstenberg--Kesten theorem implies that, for every invariant measure $\mu$ the limits
\[
  \chi^+(x)=\lim_{n\to\infty}\frac1n\log\|Dg_x^n\|,
  \qquad
  \chi^-(x)=-\lim_{n\to\infty}\frac1n
                   \log\|(Dg_x^n)^{-1}\|
\]
exist for $\mu$-almost every $x$, and $\chi^-(x)\leq\chi^+(x)$. Points for which these limits exists are called
\emph{Lyapunov regular}. When the exponents are distinct, due Oseledet's theorem, there is a
$Dg$-invariant splitting
\[
  T_xM=E^-(x)\oplus E^+(x),
  \qquad
  \lim_{n\to\pm\infty}\frac1n\log\|Dg_x^n v\|
       =\chi^\sigma(x)
       \quad\text{for }v\in E^\sigma(x)\setminus\{0\},
       \quad \sigma\in\{-,+\}.
\]
The exponents are invariant functions of $x$ and hence are constant almost everywhere when $\mu$ is ergodic; their values are then denoted
by $\chi^-(\mu)$ and $\chi^+(\mu)$. We call $\mu$ \emph{hyperbolic of saddle type} if $\chi^-(x)<0<\chi^+(x)$ almost everywhere. All
hyperbolic measures used in this paper are of saddle type. 

Positive entropy ensures this condition on a surface. Indeed, Ruelle's inequality \cite{Ruelle1978}, applied to $g$ and $g^{-1}$, gives, for an
ergodic measure with $h_\mu(g)>0$,
\begin{equation}\label{eq:prelim-ruelle}
  \chi^-(\mu)\leq-h_\mu(g)<0<h_\mu(g)\leq\chi^+(\mu).
\end{equation}
In particular, every ergodic positive-entropy measure $\nu$ of $\Phi_g$ admits an ergodic invariant lift $\mu$ under $H_g$, and every
such lift is hyperbolic: existence follows as in \cite[Proposition~2.13]{CHHU2026}, while $h_\mu(g)\geq h_\nu(\Phi_g)>0$ and
\eqref{eq:prelim-ruelle} give hyperbolicity. This is the ergodic case of \cite[Proposition~2.15]{CHHU2026}.

Let $\mathcal R$ denote the set of regular points with
$\chi^-(x)<0<\chi^+(x)$. For $x\in\mathcal R$, the \emph{Pesin stable
and unstable manifolds} are
\[
  W^s(x;g)=\left\{y\in M:
      \limsup_{n\to\infty}\frac1n
       \log d_M(g^n x,g^n y)<0\right\},
  \qquad W^u(x;g)=W^s(x;g^{-1}).
\]
Pesin's stable manifold theorem says that these sets are injectively immersed curves tangent at $x$ to $E^-(x)$ and $E^+(x)$, respectively;
see \cite[Supplement~S]{KatokHasselblatt1995}. They are invariant in the sense that $g(W^\sigma(x;g))=W^\sigma(gx;g)$ for
$\sigma\in\{s,u\}$. We use their intrinsic topology and arclength metric. A \emph{local plaque} is a sufficiently small embedded
interval in one of these curves containing its base point. The two \emph{half-leaves} at $x$ are the components of $W^\sigma(x;g)\setminus\{x\}$ in the intrinsic topology. For a lift $\widetilde x$ of $x$, we use the corresponding notation $\widetilde W^\sigma(\widetilde x;G)$ upstairs.

For each hyperbolic invariant measure, a set of full measure is covered by countably many compact \emph{Pesin blocks}. On each block the
contraction estimates are uniform, the angle between the stable and unstable directions is bounded away from zero, and the local plaques
have a common positive lower bound on their size and depend continuously on the base point in the $C^1$ topology. Thus nearby
points in a fixed block have transverse intersections between their opposite local plaques. The blocks themselves need not be invariant.

\subsection{Katok's closing lemma and ergodic homoclinic classes}
\label{subsec:prelim-homoclinic}
Katok's closing lemma \cite[\S3]{Katok1980} turns a close return in a Pesin block into a periodic orbit with controlled invariant
manifolds. More precisely, given a block $K$ and $\varepsilon>0$, there is $\delta>0$ such that
\[
  x,g^n x\in K,\qquad d_M(x,g^n x)<\delta
\]
imply the existence of a hyperbolic saddle $p\in\operatorname{Fix}(g^n)$ satisfying
\[
  \max_{0\leq j<n}d_M(g^j p,g^j x)<\varepsilon.
\]
The closing construction also controls the local invariant manifolds of $p$. When the return occurs in a sufficiently small part of $K$,
the stable and unstable plaques of $p$ meet the opposite plaques through nearby points of $K$ transversely.

For a periodic saddle $p$, write $\mathcal O(p)$ for its orbit and set
\[
  W^\sigma(\mathcal O(p);g)
      =\bigcup_{z\in\mathcal O(p)}W^\sigma(z;g),
      \qquad \sigma\in\{s,u\}.
\]
Following \cite[\S1.1]{HHTU2011}, its \emph{ergodic homoclinic class}
is
\begin{equation}\label{eq:prelim-ehc}
  \operatorname{Ehc}(p)=
  \left\{x\in\mathcal R:
    \begin{array}{l}
      W^s(x;g)\pitchfork W^u(\mathcal O(p);g)\ne\varnothing,\\
      W^u(x;g)\pitchfork W^s(\mathcal O(p);g)\ne\varnothing
    \end{array}\right\}.
\end{equation}
Here $\pitchfork$ means that an intersection exists at which the two tangent lines are distinct. This is a measurable $g$-invariant set.
Two periodic saddles $p$ and $q$ are \emph{homoclinically related}, written $p\sim_h q$, if the stable manifold of each orbit meets the
unstable manifold of the other transversely. This is an equivalence relation, and $p\sim_h q$ implies $\operatorname{Ehc}(p)=\operatorname{Ehc}(q)$.

These facts use the inclination lemma: an arc crossing the stable manifold of a periodic saddle transversely has forward iterates
containing subarcs that converge in the $C^1$ topology to any prescribed compact arc of its unstable manifold. Iterates may be
taken to be multiples of the period. The analogous assertion with stable and unstable interchanged follows by applying the lemma to
$g^{-1}$.

The following consequence of closing and recurrence expresses the connection between an ergodic hyperbolic measure and these classes;
see \cite[\S\S3--4]{Katok1980}, \cite{HHTU2011} and \cite{BCS2022}.

\begin{proposition}\label{prop:prelim-homoclinic-approximation}
Let $\mu$ be an ergodic hyperbolic measure. There is a periodic saddle $p$ such that
\[
  \mu(\operatorname{Ehc}(p))=1,
  \qquad
  \operatorname{supp}\mu\subset
    \overline{\{q:q\text{ is a periodic saddle and }q\sim_h p\}}.
\]
\end{proposition}
\begin{proof}
Choose a sufficiently small positive-measure part of a Pesin block and close a recurrent orbit in it. The plaque intersections supplied
by the closing lemma place this part of the block in $\operatorname{Ehc}(p)$. Invariance and ergodicity then give $\mu(\operatorname{Ehc}(p))=1$. If an open set $U$ meets $\operatorname{supp}\mu$, a typical orbit starting in the same part of the block visits $U$ before returning to that part. Closing a
sufficiently long such segment, with sufficiently small error, produces a periodic orbit meeting $U$. Its plaques cross those of
$p$ in both directions, so it is homoclinically related to $p$.
\end{proof}
We remark that the condition $\mu(\operatorname{Ehc}(p))=1$ does not assert that this class carries only one invariant ergodic measure.

\subsection{Conditional measures and leafwise injectivity}
\label{subsec:prelim-conditionals}

Fix an ergodic hyperbolic measure $\mu$ with $h_\mu(g)>0$. We recall how its conditional measures describe the distribution of $\mu$ along Pesin manifolds, and how entropy equality over $H_g$ constrains this distribution.

A measurable partition is understood to be countably generated modulo $\mu$. For such a partition $\xi$, let $\xi(x)$ denote
the atom containing $x$. Rokhlin's disintegration theorem \cite{Rokhlin1949} provides a measurable family of conditional
probabilities $\{\mu_x^\xi\}$, unique modulo $\mu$, which is constant along atoms and satisfies
\[
  \mu_x^\xi(\xi(x))=1
  \quad\text{for $\mu$-almost every }x,
  \qquad
  \mu(A)=\int_M\mu_x^\xi(A)\,d\mu(x)
\]
for every Borel set $A\subset M$. In particular, if $\mu(X)=1$, then $\mu_x^\xi(X)=1$ for $\mu$-almost every $x$. If a measurable partition $\eta$ refines $\xi$, its conditional entropy is
\[
  H_\mu(\eta\mid\xi)
    =-\int_M\log\mu_x^\xi(\eta(x))\,d\mu(x).
\]

A partition $\xi^u$ is \emph{subordinate to the unstable manifolds} if, for $\mu$-almost every $x$, the atom $\xi^u(x)$ is contained
in $W^u(x;g)$ and contains a neighborhood of $x$ in its intrinsic topology. It is \emph{increasing} if $g^{-1}\xi^u$ refines $\xi^u$,
where
\[
  (g^{-1}\xi^u)(x)=g^{-1}\bigl(\xi^u(gx)\bigr).
\]

Increasing partitions subordinate to the unstable manifolds will be called \emph{adapted}. By \cite[Proposition~3.1]{LedrappierStrelcyn1982}, for any hyperbolic invariant measure one can find arbitrarily small adapted partitions $\xi^u$ so that
\begin{align*}
 &\bigvee_{n\geq0}g^{-n}\xi^u=\varepsilon\quad\text{modulo }\mu\\
 &W^u(x;g)=\bigcup_{n\geq0}g^n\bigl(\xi^u(g^{-n}x)\bigr)
 \end{align*} 
for $\mu$-almost every $x$. Here $\varepsilon$ is the partition into points, and $\vee$ denotes common refinement.

For these adapted partitions, one can compute the entropy as
\begin{equation}\label{eq:prelim-subordinate-entropy}
  h_\mu(g)
    =H_\mu(g^{-1}\xi^u\mid\xi^u)
    =-\int_M
       \log\mu_x^{\xi^u}\bigl((g^{-1}\xi^u)(x)\bigr)\,d\mu(x);
\end{equation}
see \cite{Rokhlin1967,LedrappierStrelcyn1982}. Applying the same construction to $g^{-1}$ gives a partition $\xi^s$ subordinate to the stable manifolds.

Since $h_\mu(g)>0$, the conditional measures of $\mu$ along stable and unstable plaques are nonatomic; see \cite[Proposition~7.3.1 and Theorem~C$'$]{LedrappierYoung1985II}, applied to $g$ and $g^{-1}$.\par

Note that for these conditional measures, almost every point is accumulated by points in the support from both sides. Indeed. the points
of the support that fail this property belong to the countable set of endpoints of complementary intervals.

Let $\nu=(H_g)_*\mu$, and let $\mathcal F_g$ be the $\mu$-completion of $H_g^{-1}\mathcal B_{L(g)}$: note that $\mathcal F_g$ is a $g$-invariant
$\sigma$-algebra. The entropy of $g$ relative to the quotient is defined by
\[
  h_\mu(g\mid\mathcal F_g)=\sup_{\eta}\lim_{n\to\infty}\frac1n H_\mu\left(\bigvee_{j=0}^{n-1}g^{-j}\eta\,\middle|\,\mathcal F_g\right),
\]
where the supremum is over finite measurable partitions $\eta$. It measures the entropy remaining after conditioning on the
quotient coordinate. 

The Abramov--Rokhlin formula \cite{AbramovRokhlin1962} gives 
\begin{equation}\label{eq:prelim-relative-entropy}
  h_\mu(g)
    =h_\nu(\Phi_g)+h_\mu(g\mid\mathcal F_g).
\end{equation}
In particular, $h_\mu(g)=h_\nu(\Phi_g)$ is equivalent to vanishing relative entropy.

Under this equality, the argument of \cite[Proposition~4.2]{CHHU2026} gives injectivity of $H_g$ along Pesin manifolds on a set of full measure. We recall the partition argument, which applies to the present quotient dynamics situation.

Let $\mathcal P_H$ be the partition into fibers of $H_g$, and set
\[
  \zeta=\xi^u\vee\mathcal P_H.
\]
The fiber partition is invariant. Vanishing relative entropy, applied to $g^{-1}$, implies that every finite measurable partition $\eta$ satisfies
\[
  H_\mu\left(
     \eta\,\middle|\,
     \bigvee_{n\geq1}g^n\eta\vee\mathcal P_H
  \right)=0.
\]
Thus $\eta$ is measurable, modulo $\mu$, with respect to the partition on the right. Approximating $\xi^u$ by an
increasing sequence of finite partitions and using that $g\xi^u$ is coarser than $\xi^u$, we obtain
\[
  \xi^u\vee\mathcal P_H
       =g\xi^u\vee\mathcal P_H
       \quad\text{modulo }\mu.
\]
Hence $g^{-1}\zeta=\zeta$ modulo $\mu$. Since $\zeta$ refines $\xi^u$ and the partitions $g^{-n}\xi^u$ generate $\varepsilon$, it follows that $\zeta=\varepsilon$ modulo $\mu$. Equivalently,
\[
  \mu_x^\zeta=\delta_x
  \quad\text{for $\mu$-almost every }x.
\]
There is, therefore, a conull Borel set on which $H_g$ is injective within each atom of $\xi^u$.

Applying the same argument to $\xi^s$, taking the intersection of all iterates of the resulting conull sets, and using exhaustion of the global Pesin manifolds, we obtain an invariant Borel set $X$ with $\mu(X)=1$ such that
\begin{equation}\label{eq:prelim-leafwise-injectivity}
  H_g\big|_{X\cap W^\sigma(x;g)}
      \text{ is injective},
  \qquad x\in X,\quad \sigma\in\{s,u\}.
\end{equation}
The argument concerns measurable partitions on $M$ and the invariant partition into fibers of $H_g$; it requires no manifold structure
on $L(g)$.

\subsection{Product structure and the Sard argument}
\label{subsec:prelim-holonomy}

Let $g$ be a $C^\infty$ surface diffeomorphism and let $\mu$ be an ergodic hyperbolic measure of maximal entropy. We recall
how the Sard argument and Ben Ovadia's product structure are combined in \cite[Proposition~4.5]{CHHU2026}.

By \cite[Theorems~5.7 and~6.1]{BO2023}, the unstable conditional measures of $\mu$ form an absolutely continuous
family under local stable holonomies. In particular, let $I_1,I_2$ be unstable segments and let
$h^s\colon B\subset I_1\to I_2$ be such a holonomy, obtained by following stable manifolds. Writing $\mu_i^u$ for the
corresponding unstable conditional measures, we have
\begin{equation}\label{eq:prelim-holonomy-measures}
  \mu_1^u(B)>0
  \quad\Longrightarrow\quad
  \mu_2^u\bigl(h^s(B)\bigr)>0.
\end{equation}

The same construction provides leaf measures on local unstable manifolds of periodic saddles. Choose $p$ with $\mu(\operatorname{Ehc}(p))=1$ as in Proposition~\ref{prop:prelim-homoclinic-approximation}, and let $q\sim_h p$. On the local holonomy domains, these measures are equivalent to the images of unstable conditional measures on typical plaques and satisfy the holonomy property~\eqref{eq:prelim-holonomy-measures}. Countably many such local pieces and their periodic iterates determine a measure class on $W^u(q;g)$, preserved by periodic iteration. Choose a measure $m_q^u$ in this class; the same notation will be used for its lift to $\widetilde W^u(\widetilde q;G)$. When $h_\mu(g)>0$, this measure is nonatomic by Section~\ref{subsec:prelim-conditionals} and the local holonomy comparison.\par

Positive entropy also implies that the unstable conditional measures give zero mass to sets of zero Hausdorff dimension
\cite{LedrappierYoung1985II}. On compact local holonomy domains in Pesin blocks, the stable holonomies are bi-Lipschitz; see \cite[\S2.2]{CHHU2026}. Together with the measure comparisons above, this gives the same zero-mass property for $m_q^u$.

We make explicit how the same comparison transfers conull sets. Let $A$ be an invariant $\mu$-conull set of regular points, and let $I$ be an interval in a typical unstable manifold on which $A$ has full conditional measure. Here the conditional measures are extended along the manifold by iteration, using the exhaustion in Section~\ref{subsec:prelim-conditionals}. Fix $k$ with $g^kq=q$. There is a set $A_q$ of full $m_q^u$-measure in $W^u(q;g)$, invariant under $g^k$, such that each $y\in A_q$ lies on a stable manifold through $A$ which meets $W^u(q;g)$ transversely at $y$.
Moreover, every transverse intersection of this stable manifold with $I$ belongs to $A$.

Indeed, disintegration and the local holonomy comparisons give the first assertion on each of the countably many pieces defining the measure class of $m_q^u$. For the second assertion, the Pesin blocks, their iterates, and a countable collection of local charts cover the transverse intersections with $I$ by countably many stable holonomies $h_j$, with domains in $W^u(q;g)$. Since $I\setminus A$ has zero conditional measure,
the holonomy property~\eqref{eq:prelim-holonomy-measures} gives
\[
  m_q^u\bigl(h_j^{-1}(I\setminus A)\bigr)=0
  \qquad\text{for every }j.
\]
Remove these exceptional sets, the exceptional set for the first assertion, and all their forward and backward $g^k$-iterates. Their union is
still $m_q^u$-null because periodic iteration preserves the measure class. Its complement is the required invariant set $A_q$. The same conclusion holds for lifts.

The existence of the required transverse intersections is provided by the Sard argument developed in \cite{HHTU2011SRB} for SRB measures, later adapted to measures of maximal entropy in \cite{BCS2022}; see \cite[\S2.2]{CHHU2026}. In the configuration considered in \cite[Proposition~4.5]{CHHU2026}, a positive conditional-measure family of stable manifolds topologically crosses two disjoint unstable segments. These crossings may initially
be tangential. The Sard argument gives a positive conditional-measure subset for which the intersections with both segments are transverse. Restricting this subset provides a stable holonomy with a domain of positive conditional measure. Ben Ovadia's result then ensures that
its image also has positive conditional measure.

\subsection{Entropy, horseshoes, and semistability}
\label{subsec:prelim-continuity}

In this subsection we enumerate the preliminaries for Theorems~\ref{thm:entropy.formula} and \ref{thm:quotient-entropy-continuity}.
We recall the continuity properties of the quotient maps, the entropy formulas that will be used, and the construction and 
$C^0$ semistability of hyperbolic horseshoes. We retain the notation of Section~\ref{sec:hhu-preliminaries} and the conventions concerning hyperbolic
measures and Pesin blocks from \Cref{subsec:prelim-pesin}.

\medskip
\noindent\textbf{(i) Continuity of the quotient maps.}
As recalled in Section~\ref{subsec:hhu-orbit-quotient}, the semiconjugacies $H_g\colon M\to Z$ depend continuously on $g$ in the $C^0$
topology, and the sets $L(g)$ depend continuously on $g$ in the Hausdorff topology \cite[Theorem~1.2 and \S2]{HHU2004}.
Thus, if homeomorphisms $g_j$ in the isotopy class of $f$ converge uniformly to $g$, then 
\begin{equation}\label{eq:prelim-quotient-continuity}
  H_{g_j}\longrightarrow H_g
  \quad\text{uniformly},
  \qquad
  L(g_j)\longrightarrow L(g)
  \quad\text{in the Hausdorff topology}.
\end{equation}
All these sets are invariant under the same ambient map $\Phi$.
Their union
\[
  K=L(g)\cup\bigcup_{j\geq1}L(g_j)
\]
is compact. The hyperbolic local coordinates recalled in Section~\ref{subsec:hhu-ambient-map} imply that $\Phi$ is expansive on every compact invariant subset of $Z$. Thus, one can choose a single expansivity constant on $K$.

\medskip
\noindent\textbf{(ii) Entropy and expansivity.}
Let $T\colon K\to K$ be an expansive homeomorphism of a compact metric space, with expansivity constant $\delta>0$ (hence, two points whose iterates remain at distance less than $\delta$ for every integer time must coincide). We recall the Adler--Konheim--McAndrew definition of topological entropy \cite{Adler_1965}. Let $A \subset K$ be a compact $T$-invariant subset, and consider $\mathcal U$ a finite open cover of $K$ with mesh smaller than $\delta$. Given $n\geq 0$, write
\[
  \mathcal U^{(n)}=\bigvee_{i=0}^{n-1}T^{-i}\mathcal U,
\]
where the join consists of the intersections $U_0\cap T^{-1}U_1\cap\cdots\cap T^{-(n-1)}U_{n-1}$, with $U_i\in\mathcal U$. Then
\begin{equation}\label{eq:prelim-cover-entropy}
  \htop(T|_A)=\lim_{n\to\infty}\frac1n\log N(\mathcal U^{(n)}|_A)=\inf_{n\geq1}\frac1n\log N(\mathcal U^{(n)}|_A),
\end{equation}
where $N(\mathcal V|_A)$ denotes the least number of members of $\mathcal V$ needed to cover $A$. 

We will also use the variational principle in the form
\[
  \htop(T)
  =\sup\{h_\nu(T):\nu
       \text{ is an ergodic }T\text{-invariant probability measure}\}.
\]
See \cite{ErgWal}. 

We note that expansivity gives uniqueness of sufficiently close semiconjugacies. Indeed, if $S\colon X\to X$ is a homeomorphism, and  $a,b\colon X\to K$ satisfy $a\circ S=T\circ a, b\circ S=T\circ b$ with
\[
   \sup_{x\in X}d(a(x),b(x))<\delta
\]
then
 \[
   \forall n,\ d(T^na(x),T^nb(x))=d(a(S^kx),b(S^kx))<\delta\Rightarrow a(x)=b(x).
\]
We remind the reader that if $T$ is expansive, then each of its iterates is also expansive.

\medskip
\noindent\textbf{(iii) The Brin--Katok formula and separated returns.}
For a homeomorphism $T$ of a compact metric space $(Y,d_Y)$, denote
\begin{align*}
&d_n^T(y,z)=\max_{0\leq i<n}d_Y(T^iy,T^iz)\\
&B_n^T(y,r)=\{z\in Y:d_n^T(y,z)<r\}.
\end{align*}
A set is $(n,r)$-separated if distinct points have $d_n^T$-distance at least $r$.

For an ergodic invariant probability measure $\nu$ of finite entropy, the Brin--Katok local entropy formula \cite{BK1983}
states that, for $\nu$-almost every $y$,
\begin{equation}\label{eq:prelim-brin-katok}
\begin{split}
  h_\nu(T)
  &=\lim_{r\downarrow0}\liminf_{n\to\infty}
      -\frac1n\log\nu(B_n^T(y,r))\\
  &=\lim_{r\downarrow0}\limsup_{n\to\infty}
      -\frac1n\log\nu(B_n^T(y,r)).
\end{split}
\end{equation}

For every $\epsilon>0$, this gives a radius $r>0$, an integer $n_0\geq1$, and a measurable set $A\subset Y$ of positive $\nu$-measure such that
\begin{equation}\label{eq:prelim-uniform-bowen-bound}
  \nu(B_n^T(y,r))
  \leq e^{-n(h_\nu(T)-\epsilon)}
  \qquad(y\in A,\ n\geq n_0).
\end{equation}
To obtain this uniform estimate, first fix a sufficiently small radius in \eqref{eq:prelim-brin-katok}, and then restrict to a
positive-measure set on which the lower bound on $n$ is uniform.

We will apply this estimate to $T=\Phi_g$ and $\nu=(H_g)_*\mu$, where $\mu$ is an ergodic invariant lift. Recall from \Cref{subsec:prelim-pesin} that such a lift exists and is hyperbolic when $g$ is $C^{1+\alpha}$ and $h_\nu(\Phi_g)>0$.

\begin{lemma}\label{lem:prelim-separated-returns}
Given $B\subset H_g^{-1}(A)$ with $\mu(B)>0$, there exists a relatively dense set of integers $n$ and subsets $E_n \subset B\cap g^{-n}B$ so that
\begin{itemize}
  \item $\#E_n\geq \frac{\mu(B)^2}{2}e^{n(h_\nu(\Phi_g)-\epsilon)}$,
  \item $H_g(E_n)$ is $(n,r)$-separated.
\end{itemize}
\end{lemma}

\begin{proof}
  Khintchine's recurrence theorem gives a syndetic set of integers $n$ so that $\mu(B\cap g^{-n}B)\geq \frac{\mu(B)^2}{2}:=c_B$. At each one of these times choose a maximal $(n,r)$-separated subset of $H_g(B\cap g^{-n}B)$, and let $E_n \subset B\cap g^{-n}B$ by a section under $H_g$. Then 
  \[
  H_g(B\cap g^{-n}B) \subset \bigcup_{x\in E_n} B_n^T(x,r)
  \]
  so, by \eqref{eq:prelim-uniform-bowen-bound},
\[
  c_B
  \leq\mu(B\cap g^{-n}B)
  \leq\sum_{x\in E_n}\nu(B_n^{\Phi_g}(H_g(x),r))
  \leq\#E_n\,e^{-n(h_\nu(\Phi_g)-\epsilon)}\Rightarrow \#E_n\geq c_Be^{n(h_\nu(\Phi_g)-\epsilon)},
\]
\end{proof}

The set $B$ will be chosen after the radius $r$, which allows the returns to be placed in a sufficiently small Pesin rectangle.

\medskip
\noindent\textbf{(iv) Hyperbolic return rectangles.}
The Katok--Mendoza construction turns returns to a small part of a Pesin block into branches of a hyperbolic horseshoe;
see \cite[Theorem~S.4.16 and the proof of Theorem~S.5.9]{KatokHasselblatt1995}. We recall the geometric form of the construction.

Fix a Lyapunov chart, and take the stable direction to be vertical, and the unstable direction to be horizontal. A vertical strip joins
the bottom to the top of a rectangle and is bounded on its left and right by graphs over the vertical coordinate. A horizontal strip is defined by interchanging the coordinates. By reducing the size of the domain, these  boundary graphs in the construction can be taken to have slopes uniformly smaller than one.

Let $g$ be $C^{1+\alpha}$ and let $\mu$ be an ergodic hyperbolic measure. Given $\eta>0$ and a measurable set $A_0\subset M$ with $\mu(A_0)>0$, the regular-rectangle construction provides a rectangle $R$ and a positive-measure set $B\subset A_0$ in the central part of $R$ and in a fixed
Pesin block. The restriction to $A_0$ is possible because regular rectangles cover compact regular sets of arbitrarily
large measure; one of their central parts must meet $A_0$ in positive measure.

For every sufficiently large $n$ and every $x\in B\cap g^{-n}B$, there is a compact vertical strip $V_x\subset R$ containing $x$ such that $g^n(V_x)$ is a horizontal strip crossing $R$ and
\begin{equation}\label{eq:prelim-return-accuracy}
  \operatorname{diam}_{d_M}(g^iV_x)<\eta,
  \qquad 0\leq i\leq n.
\end{equation}
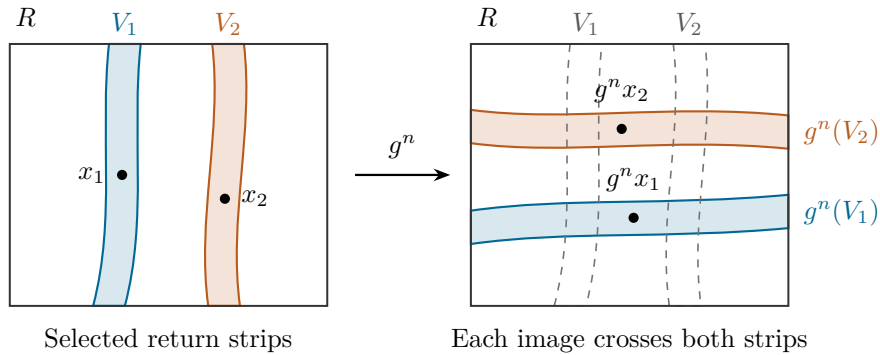
\begin{figure}[htbp]
\centering
\begin{tikzpicture}[
  x=1.05cm,y=1.05cm,
  font=\small,
  frame/.style={draw=black!80,line width=.75pt},
  original/.style={draw=black!55,dashed,line width=.6pt},
  point/.style={circle,fill=black,inner sep=1.35pt},
  map/.style={-{Stealth[length=2.2mm]},line width=.85pt},
  every node/.style={inner sep=2pt}
]
\definecolor{KMblue}{RGB}{0,104,148}
\definecolor{KMorange}{RGB}{184,91,28}

\def\Vone{
  (1.05,0) .. controls (1.34,1.10) and (1.13,2.20) .. (1.25,3.30)
  -- (1.65,3.30)
  .. controls (1.53,2.20) and (1.74,1.10) .. (1.45,0) -- cycle}
\def\Vtwo{
  (2.50,0) .. controls (2.35,1.10) and (2.68,2.20) .. (2.55,3.30)
  -- (2.95,3.30)
  .. controls (3.08,2.20) and (2.75,1.10) .. (2.90,0) -- cycle}

\path[fill=KMblue!15,draw=KMblue,line width=.8pt] \Vone;
\path[fill=KMorange!17,draw=KMorange,line width=.8pt] \Vtwo;
\draw[frame] (0,0) rectangle (4,3.30);
\node[anchor=south west] at (0,3.45) {$R$};
\node[KMblue,above] at (1.45,3.34) {$V_1$};
\node[KMorange,above] at (2.75,3.34) {$V_2$};
\node[point,label={[xshift=-2pt]left:$x_1$}] at (1.415,1.65) {};
\node[point,label={[xshift=2pt]right:$x_2$}] at (2.71,1.35) {};
\node at (2,-.43) {Selected return strips};

\draw[map] (4.35,1.65) -- node[above=4pt] {$g^n$} (5.55,1.65);

\begin{scope}[xshift=6.1cm]
  \path[fill=KMblue!15,draw=KMblue,line width=.8pt]
    (0,.78) .. controls (1.20,.95) and (2.50,.84) .. (4,.98)
    -- (4,1.40)
    .. controls (2.50,1.26) and (1.20,1.37) .. (0,1.20) -- cycle;
  \path[fill=KMorange!17,draw=KMorange,line width=.8pt]
    (0,2.05) .. controls (1.20,1.89) and (2.50,2.13) .. (4,1.98)
    -- (4,2.40)
    .. controls (2.50,2.55) and (1.20,2.31) .. (0,2.47) -- cycle;
  \draw[original] \Vone;
  \draw[original] \Vtwo;
  \draw[frame] (0,0) rectangle (4,3.30);
  \node[anchor=south west] at (0,3.45) {$R$};
  \node[black!65,above] at (1.45,3.34) {$V_1$};
  \node[black!65,above] at (2.75,3.34) {$V_2$};
  \node[point] at (2.05,1.11) {};
  \node[above=9pt,inner sep=1pt] at (2.05,1.11)
    {$g^n x_1$};
  \node[point] at (1.90,2.23) {};
  \node[above=9pt,inner sep=1pt] at (1.90,2.23)
    {$g^n x_2$};
  \node[KMblue,anchor=west] at (4.12,1.18) {$g^n(V_1)$};
  \node[KMorange,anchor=west] at (4.12,2.19) {$g^n(V_2)$};
  \node at (2,-.43) {Each image crosses both strips};
\end{scope}
\end{tikzpicture}
\caption{Katok--Mendoza return rectangles. Both panels depict the same rectangle $R$, with the stable direction vertical and the unstable
direction horizontal. The points $x_i$ and $g^n x_i$ lie in the chosen return set $B$. Each vertical strip $V_i$ follows the return segment
of $x_i$ and is mapped to a horizontal strip. The dashed boundaries in the right panel mark the original vertical strips. Each horizontal image crosses both vertical strips, so all transitions between the two selected returns are possible.}
\label{fig:katok-mendoza-returns}
\end{figure}

Thus every point in $V_x$ follows the selected return segment with the prescribed accuracy. The vertical strips lie away from the left and right sides
of $R$, and their horizontal images lie away from its top and bottom sides. The return maps preserve and expand cones about the horizontal direction; their inverses preserve and expand cones about the vertical direction, with uniform estimates. The rectangle and $B$ are fixed before choosing the return time.

Any finite collection of $N\geq2$ pairwise disjoint such strips, all with the same return time $n$, defines a horseshoe
\begin{equation}\label{eq:prelim-return-horseshoe}
  \Lambda
  =\bigcap_{k\in\mathbb Z}g^{-kn}
       \left(\bigcup_{i=1}^{N}V_i\right).
\end{equation}
This is a compact, locally maximal hyperbolic set for $g^n$. Every horizontal image crosses every vertical strip, so any
selected return can be followed by any other. The hyperbolic estimates give a unique point for each bi-infinite itinerary.
Consequently, the itinerary map conjugates $g^n|_\Lambda$ to the full two-sided shift on $N$ symbols, and
\begin{equation}\label{eq:prelim-horseshoe-entropy}
  \htop(g^n|_\Lambda)=\log N.
\end{equation}
The orbit estimate \eqref{eq:prelim-return-accuracy} is the feature of this construction that will be essential below.

We shall select return segments whose images under $H_g$ are separated, as in \ref{lem:prelim-separated-returns}. By uniform continuity of $H_g$, the strips can be chosen so that their points follow these segments closely enough to preserve separation in the quotient. Whenever two horseshoe points have different itineraries, their images under $H_g$ can therefore be distinguished at an iterate where the itineraries differ. This will give injectivity of $H_g$ on the horseshoe, so that its entropy is also the entropy of its image in $L(g)$. The semistability statement recalled next will allow us to follow this horseshoe under $C^0$ perturbations.

\medskip
\noindent\textbf{(v) $C^0$ semistability.}
The horseshoes just described persist as factors under small $C^0$ perturbations. More precisely, Nitecki's semistability
theorem \cite{Nitecki1971}, applied to such a horseshoe $\Lambda$ for $T=g^n$, gives the following statement. If homeomorphisms $T_j\colon M\to M$ converge uniformly to $T$, then, for all sufficiently large $j$, there are compact $T_j$-invariant sets $\Lambda_j$ and continuous surjections
$r_j\colon\Lambda_j\to\Lambda$ satisfying
\begin{equation}\label{eq:prelim-horseshoe-semistability}
  r_j\circ T_j=T\circ r_j,
  \qquad
  \sup_{z\in\Lambda_j}d_M(z,r_j(z))\longrightarrow0.
\end{equation}

For the strip construction above, $\Lambda_j$ can be taken to be the set of points whose complete $T_j$-orbits remain in $\bigcup_iV_i$. The map $r_j$ assigns to each such point the unique point of $\Lambda$ with the same itinerary. The crossings have positive margins and survive uniform perturbations, so every itinerary of the original horseshoe is still realized. This explains the surjectivity of $r_j$ and why the statement applies to perturbing homeomorphisms without differentiability assumptions.

\section{Entropy formula}\label{sec:entropy.formula}
Recall that $h_N(g)$ denotes the exponential growth of non-empty Nielsen classes of $g^n$ (cf. \eqref{eq:Nielsengrowth}), 
\[
  h_N(g)=\limsup_{n\to\infty}\frac1n\log\max\{1,C_n(g)\}.
\]

\begin{proof}[Proof of \Cref{thm:entropy.formula}]
By \Cref{subsec:hhu-local-structure}, two points of $\operatorname{Fix}(g^n)$ have the same image under $H_g$
if and only if they are Nielsen equivalent. Hence
\[
  C_n(g)=\#H_g(\operatorname{Fix}(g^n))\leq\#\operatorname{Fix}(\Phi_g^n).
\]
Moreover, expansivity implies that $\operatorname{Fix}(\Phi_g^n)$ is $(n,\delta)$-separated for a fixed $\delta>0$. It follows that
\[
  h_N(g)
  \leq\limsup_{n\to\infty}\frac1n
       \log\max\{1,\#\operatorname{Fix}(\Phi_g^n)\}
  \leq\htop(\Phi_g).
\]
For the reverse inequality, assume first that $g$ is $C^{1+\alpha}$. Let $\nu$ be an ergodic $\Phi_g$-invariant probability measure with positive entropy, and consider an ergodic hyperbolic lift $\mu$ under $H_g$. Fix $\epsilon>0$. Choose the radius $r>0$ and the positive-$\nu$-measure set $A\subset L(g)$ given by \eqref{eq:prelim-uniform-bowen-bound} in \Cref{subsec:prelim-continuity}, item~(iii). We shall select our return points in $H_g^{-1}(A)$, so that the separated-return estimate applies to their images. By uniform continuity of $H_g$, choose $\eta>0$
such that
\[
  d_M(a,b)<\eta
  \quad\Longrightarrow\quad
  d(H_g(a),H_g(b))<r/3.
\]
Choose a Pesin block $K$ satisfying $\mu(K\cap H_g^{-1}(A))>0$. For this block and the accuracy $\eta$, Katok's closing
lemma, recalled in \Cref{subsec:prelim-homoclinic}, provides a closing radius $\rho>0$.
Choose a measurable set
\[
  B\subset K\cap H_g^{-1}(A),
  \qquad
  \mu(B)>0,
  \qquad
  \operatorname{diam}_{d_M}(B)<\rho.
\]

The separated-return estimates given in \ref{lem:prelim-separated-returns}, applied to $B$, provides arbitrarily large $n$ and finite sets $E_n\subset B\cap g^{-n}B$ such that $H_g$ is injective on $E_n$, the set $H_g(E_n)$ is $(n,r)$-separated, and
\[
  \#E_n\geq \frac{\mu(B)^2}{2} e^{n(h_\nu(\Phi_g)-\epsilon)}.
\]
For each $x\in E_n$, both $x$ and $g^n x$ belong to $B$. Our choice of $B$ therefore allows us to close this return: there is $p_x\in\operatorname{Fix}(g^n)$ with
\[
  \max_{0\leq j<n}d_M(g^j p_x,g^j x)<\eta.
\]
The choice of $\eta$ and the semiconjugacy identity give
\[
  d_n^{\Phi_g}\bigl(H_g(p_x),H_g(x)\bigr)<r/3.
\]
Since the images of distinct points of $E_n$ are $(n,r)$-separated, we get
\[
  d_n^{\Phi_g}\bigl(H_g(p_x),H_g(p_y)\bigr)>r/3
  \qquad(x\ne y).
\]
Thus the points $H_g(p_x)$ are distinct, and the periodic points $p_x$ belong to distinct Nielsen classes. We obtain
\[
  C_n(g)\geq\#E_n
  \geq \frac{\mu(B)^2}{2} e^{n(h_\nu(\Phi_g)-\epsilon)}.
\]
Taking the limsup gives $h_N(g)\geq h_\nu(\Phi_g)-\epsilon$. Letting $\epsilon\to0$ and applying the variational
principle we get
\[
  h_N(g)\geq\htop(\Phi_g).
\]
Now we consider the case when $g$ is a $C^1$ Axiom~A diffeomorphism. Note that by the variational principle, every invariant probability measure is supported on on the uniformly hyperbolic nonwandering set $\Omega(g)$. In the preceding construction, take $K=\Omega(g)$ and use the uniform hyperbolic closing lemma (see for example \cite{KatokHasselblatt1995}) to choose $\rho$. The same separated returns then produce periodic points in distinct Nielsen classes, proving the reverse inequality in this case as well.

Together with the initial inequalities, this proves all three equalities.
\end{proof}
\begin{remark}
    The $C^{1+\alpha}$ hypothesis above is only used to apply Katok's closing lemma. 
\end{remark}

\begin{proof}[Proof of \Cref{thm:strict-entropy-gap}]
Suppose, to the contrary, that
\[
  h_{\mathrm{top}}(g)=h_{\mathrm{top}}(\Phi_g)>h_{\mathrm{top}}(f).
\]

\Cref{thm:entropy.formula} relates the entropy of the quotient to the exponential growth of Nielsen classes. When this entropy exceeds that of $f$,
the excess growth gives rise to periodic points outside $L(f)$. The lifted tree coordinates of these new periodic points are endpoints (\Cref{subsec:hhu-local-structure}). These endpoints provide the geometric obstruction to entropy equality: we will use them to produce a fold in the tree image of an unstable  manifold, contradicting the leafwise injectivity of $H_g$ established in \Cref{subsec:prelim-conditionals} for the case when the relative entropy over $H_g$ vanishes, as it does for an invariant lift of a maximal-entropy measure of $\Phi_g$
under the assumption $\htop(\Phi_g)=\htop(g)$.\par
\smallskip
Since $\Phi_g$ is expansive on the compact space $L(g)$, its entropy map is upper semicontinuous and it has an ergodic measure of maximal
entropy $\nu$. By \Cref{subsec:prelim-pesin}, $\nu$ admits an ergodic invariant lift $\mu$ under $H_g$. The factor inequality and our assumption give
\[
  h_\nu(\Phi_g)\leq h_\mu(g)\leq h_{\mathrm{top}}(g)
     =h_\nu(\Phi_g).
\]
Thus $\mu$ is a hyperbolic measure of maximal entropy, again by \Cref{subsec:prelim-pesin}. Moreover, $L(f)$ is invariant, so ergodicity gives
$\nu(L(f))\in\{0,1\}$. The value $1$ would force $h_\nu(\Phi_g)=h_\nu(\Phi|_{L(f)})\leq h_{\mathrm{top}}(\Phi|_{L(f)})
=h_{\mathrm{top}}(f)$, contrary to the choice of $\nu$. Recalling that $J_g=H_g^{-1}(L(f))$ and $(H_g)_*\mu=\nu$, we obtain
\[
\mu(J_g)=\nu(L(f))=0.
\]
By \Cref{subsec:prelim-conditionals}, equations \eqref{eq:prelim-relative-entropy}--\eqref{eq:prelim-leafwise-injectivity}, the entropy equality gives
an invariant conull Borel set $X$ with the following lifted leafwise injectivity property:

\begin{equation}\label{eq:thm43-claims-injectivity}
  \left.\widetilde H_g\right|_{\widetilde X\cap
      \widetilde W^\sigma(\widetilde x;G)}
  \quad\text{is injective},
  \qquad \widetilde x\in\widetilde X,\quad \sigma\in\{s,u\}.
\end{equation}
Here $\widetilde X$ is the full lift of $X$. We may shrink $X$ by
invariant null sets whenever necessary.

\medskip
\begin{claim}\label{claim:unbounded.half.leaves}
After reducing $X$, every stable and unstable half-leaf based at a point of $\widetilde X$ has an unbounded $\widetilde H_g$-image.
\end{claim}

\begin{proof}[Proof of Claim \ref{claim:unbounded.half.leaves}]
We first identify the coordinate that varies along each manifold. Let $\widetilde x$ and $\widetilde y$ lie on the same lifted
stable manifold. Their forward iterates approach one another: 
\[
  D(G^n\widetilde x,G^n\widetilde y)\longrightarrow 0,
\]
where $D$ is the metric on $\widetilde M$ induced by the embedding $\iota$, as in Section~\ref{subsec:hhu-orbit-quotient}. Let $C$ be the uniform
displacement bound in \eqref{eq:hhu-lifted-semiconjugacy-metric}. The lifted semiconjugacy relation and the global splitting
\eqref{eq:hhu-global-splitting} give
\[
\lambda^n d_s\bigl(
       (\widetilde H_g)_s(\widetilde x),(\widetilde H_g)_s(\widetilde y)\bigr) = d_s\bigl((\widetilde H_g)_s(G^n\widetilde x),(\widetilde H_g)_s(G^n\widetilde y)\bigr) \leq 2C+D(G^n\widetilde x,G^n\widetilde y).
\]
The last inequality follows by comparing each semiconjugacy image with the corresponding point under $\iota$; each comparison contributes at most $C$. The right-hand side remains bounded, while $\lambda^n\to\infty$. The initial distance in the $s$-coordinate must therefore be zero.
Thus $(\widetilde H_g)_s$ is constant on each lifted stable
manifold.

For two points on the same lifted unstable manifold, their backward iterates approach one another. Under backward iteration it is the $u$-coordinate that expands by $\lambda$. The same estimate, with $G^{-n}$ and $d_u$, therefore shows that $(\widetilde H_g)_u$ is constant on each lifted unstable
manifold.

The two-sided accumulation property in Section~\ref{subsec:prelim-conditionals} implies that, for almost every $x$, each stable half-leaf contains points of $X$ arbitrarily close to $x$. Their images differ from the image of $x$ by~\eqref{eq:thm43-claims-injectivity}. Thus the image of
each half-leaf has positive radius, measured from $\widetilde H_g(\widetilde x)$.

Let $R^s(x)$ be the smaller of these two radii, allowing the value $\infty$. It is measurable and independent of the lift. Since the
only varying coordinate contracts by $\lambda^{-1}$, we have 
\[
  R^s(gx)=\lambda^{-1}R^s(x).
\]
If $0<R^s(x)<\infty$, this radius tends to zero under forward iteration. In particular, a point of $\{a\leq R^s\leq b\}$, with $0<a<b<\infty$, eventually leaves that set forever. Poincar\'e recurrence forces every such set to have measure zero. A countable cover of the positive finite
values therefore gives $R^s=\infty$ almost everywhere. Apply the same argument to $g^{-1}$ for the unstable half-leaves,
and discard the exceptional null sets and their iterates.
\end{proof}

Let $p_0$ be the saddle given by Proposition~\ref{prop:prelim-homoclinic-approximation}, so that $\mu(\operatorname{Ehc}(p_0))=1$. We may replace $p_0$ by any homoclinically related saddle without changing its ergodic homoclinic class, as recalled in Section~\ref{subsec:prelim-homoclinic}. We choose this replacement outside $J_g$. Since $\mu(J_g)=0$, there is a point $x\in\operatorname{supp}\mu\setminus J_g$. The set $J_g$ is closed, so a sufficiently small neighborhood of $x$ is disjoint from $J_g$. Proposition~\ref{prop:prelim-homoclinic-approximation} ensures that
saddles homoclinically related to $p_0$ occur arbitrarily close to $x$. Choose one in this neighborhood and denote it by $p$. Then
\[
  p\notin J_g,
  \qquad
  \operatorname{Ehc}(p)=\operatorname{Ehc}(p_0),
  \qquad
  \mu(\operatorname{Ehc}(p))=1.
\]
We fix this $p$ for the remainder of the theorem. 
\medskip
\begin{claim}\label{claim:periodic.saddle.related.to.p}
There is a periodic saddle $q\notin J_g$, homoclinically related to the saddle $p$, with the following property. Fix a lift $\widetilde q$ of $q$,
and let $m_q^u$ be the leaf measure introduced in Section~\ref{subsec:prelim-holonomy}. For every open interval
$I\subset\widetilde W^u(\widetilde q;G)$ containing $\widetilde q$, the two components $I^-$ and $I^+$ of
$I\setminus\{\widetilde q\}$ satisfy
\[
  m_q^u(I^-)>0
  \qquad\text{and}\qquad
  m_q^u(I^+)>0.
\]
\end{claim}

\begin{proof}[Proof of Claim \ref{claim:periodic.saddle.related.to.p}]
After intersecting $X$ with $\operatorname{Ehc}(p)\setminus J_g$, we still have an invariant conull set. In a small Pesin rectangle contained in $M\setminus J_g$, choose a recurrent point $x\in X$ at which the unstable conditional measure restricted to the Pesin block has support
on both sides. Such points have full measure in the block (recall that conditionals are non-atomic).

Choose positive conditional-measure sets on either side of $x$, away from $x$ itself. Katok's closing lemma, with the plaque
control in Section~\ref{subsec:prelim-homoclinic}, produces a periodic saddle $q$ sufficiently close to $x$ that the stable plaques through these two sets meet opposite unstable branches of $q$ transversely. The same plaque control preserves the transverse connections with the orbit of
$p$, so $q\sim_h p$. Since the rectangle lies outside $J_g$, we also have $q\notin J_g$.

By the absolute continuity of stable holonomy with respect to the leaf measures described in Section~\ref{subsec:prelim-holonomy}, the images of these
two sets on $\widetilde W^u(\widetilde q;G)$ have positive $m_q^u$-measure, one on each side of $\widetilde q$.

Choose a multiple $k$ of the period of $q$ and a lift $Q$ of $g^k$ fixing $\widetilde q$, with $k$ even if necessary so that $Q$ preserves both unstable branches. Backward iterates of compact positive-measure subsets of these two sets converge to $\widetilde q$. Periodic iteration preserves the reference measure class, so every one-sided unstable neighborhood of $\widetilde q$ has positive measure, as claimed.
\end{proof}

The next claim transfers both injectivity and the conclusion of Claim~\ref{claim:unbounded.half.leaves} to the periodic unstable manifold.

\medskip
\begin{claim}\label{claim:Yq}
The measure $m_q^u$ is nonatomic, and there is a set $Y_q$ of full $m_q^u$-measure in $\widetilde W^u(\widetilde q;G)$ on which
$(\widetilde H_g)_s$ is injective. For every $\widetilde y\in Y_q$, the stable manifold $\widetilde W^s(\widetilde y;G)$ meets
$\widetilde W^u(\widetilde q;G)$ transversely at $\widetilde y$, and both of its half-leaves based at $\widetilde y$ have
unbounded images under $\widetilde H_g$.
\end{claim}

\begin{proof}[Proof of Claim~\ref{claim:Yq}]
Nonatomicity follows from Section~\ref{subsec:prelim-conditionals}. To obtain injectivity on the whole unstable manifold, we will compare
points on a single typical unstable interval, where \eqref{eq:thm43-claims-injectivity} applies.

Since $q\sim_h p$, Section~\ref{subsec:prelim-homoclinic} gives $\mu(\operatorname{Ehc}(q))=1$. We may therefore choose a typical unstable manifold, on which $X$ has full conditional measure by Section~\ref{subsec:prelim-conditionals}, that meets the stable manifold of the orbit of $q$ transversely. Taking a suitable iterate and lift, choose a compact interval $I_*$ in this unstable manifold which crosses $\widetilde W^s(\widetilde q;G)$ transversely in its interior. Since $X$ is invariant, $I_*\cap\widetilde X$ has full conditional measure in $I_*$.

By the stable-holonomy argument in Section~\ref{subsec:prelim-holonomy}, applied to $A=X$ and the interval $I_*$ in the universal cover, there is
a $Q$-invariant set $Y_q$ of full $m_q^u$-measure in $\widetilde W^u(\widetilde q;G)$ with the following properties. For every $\widetilde y\in Y_q$, the stable manifold through $\widetilde y$ contains a point of $\widetilde X$ and meets $\widetilde W^u(\widetilde q;G)$ transversely at $\widetilde y$.
Moreover, every transverse intersection of this stable manifold with $I_*$ belongs to $\widetilde X$.

Choose $\widetilde x\in\widetilde X$ on the stable manifold through $\widetilde y$. In either direction along this manifold, the half-leaves starting at $\widetilde x$ and $\widetilde y$ differ only by the segment joining these two points. This segment is compact, so its image under
$\widetilde H_g$ is bounded. By Claim~\ref{claim:unbounded.half.leaves}, each half-leaf starting at $\widetilde x$ has unbounded image. Adding or
removing the bounded image of the joining segment cannot change this. Hence both half-leaves starting at
$\widetilde y$ also have unbounded images.

It remains to prove injectivity. Suppose that distinct points $\widetilde y,\widetilde z\in Y_q$ satisfy
\[
  (\widetilde H_g)_s(\widetilde y)=(\widetilde H_g)_s(\widetilde z).
\]
Choose disjoint compact stable arcs through these points. Both arcs cross $\widetilde W^u(\widetilde q;G)$ transversely. The inclination lemma from Section~\ref{subsec:prelim-homoclinic}, applied to $Q^{-1}$, shows that their backward iterates contain subarcs converging in $C^1$ to a compact segment of $\widetilde W^s(\widetilde q;G)$ crossing $I_*$. Consequently, for the same sufficiently large $n$, the two iterated arcs meet $I_*$ transversely at points $a$ and $b$. These points are distinct because the original arcs are disjoint. Since $Q^{-n}\widetilde y,Q^{-n}\widetilde z\in Y_q$, the defining property of $Y_q$ gives $a,b\in\widetilde X$.

The semiconjugacy preserves equality of the $s$-coordinate under iteration. This coordinate is also constant on stable manifolds, as shown in Claim~\ref{claim:unbounded.half.leaves}. Hence
\[
  (\widetilde H_g)_s(a)=(\widetilde H_g)_s(Q^{-n}\widetilde y)=(\widetilde H_g)_s(Q^{-n}\widetilde z)=(\widetilde H_g)_s(b).
\]
The $u$-coordinate is constant on $I_*$, which lies in a single unstable manifold. Thus $\widetilde H_g(a)=\widetilde H_g(b)$, contradicting
\eqref{eq:thm43-claims-injectivity} on $I_*\cap\widetilde X$. This proves the claim.
\end{proof}
We have now reached the geometric part of the argument. Claims~\ref{claim:periodic.saddle.related.to.p} and~\ref{claim:Yq} ensure that neither unstable branch can collapse under projection; the endpoint property forces their images to overlap.

\medskip
\begin{claim}\label{claim:4}
Writing $e=(\widetilde H_g)_s(\widetilde q)$, there are points $a<\widetilde q<b$ on $\widetilde W^u(\widetilde q;G)$ such that
\[
  (\widetilde H_g)_s(a)=(\widetilde H_g)_s(b)=t\ne e.
\]
\end{claim}
\begin{proof}[Proof of Claim \ref{claim:4}]
Since $q\notin J_g$, its periodic image $H_g(q)$ lies outside $L(f)$. By \cite[Theorem~1.4 and its proof]{HHU2004}, as recalled
in Section~\ref{subsec:hhu-local-structure}, the lifted coordinate $e$ belongs to $\widehat{\mathcal T}^s\setminus\mathcal T^s$ and is an
endpoint of the completed tree.

Each unstable branch has positive $m^u_q$-measure by Claims~\ref{claim:periodic.saddle.related.to.p} and~\ref{claim:Yq}. Its image under $(\widetilde H_g)_s$ cannot be constant, because this coordinate is injective on its intersection with $Y_q$. By continuity, the image of each branch together with $\widetilde q$ contains a nondegenerate tree arc starting at $e$.

These two arcs must share a nondegenerate initial segment. Otherwise, their union would be an arc with $e$ in its interior, contrary to the fact that $e$ is an endpoint. A point $t\ne e$ in their intersection has a preimage on each branch, giving $a$ and $b$.
\end{proof}

We now use the fold to produce repeated intersections between the stable and unstable manifolds of $G$. We will construct
a disk whose boundary contains an arc of $\widetilde W^u(\widetilde q;G)$. For a set of positive $m_q^u$-measure on this arc, the stable manifolds cross the arc transversely and avoid the remaining boundary of the disk. By Claim~\ref{claim:Yq}, the stable half-leaves entering the
disk have unbounded images under $\widetilde H_g$, so they cannot remain inside it. They must therefore cross the unstable arc again.

\medskip
\begin{claim}\label{claim:stable.unstable.loop}
There is an unstable interval $I$ of positive $m_q^u$-measure such that, for almost every $z\in I$, the stable manifold
through $z$ topologically crosses $\widetilde W^u(\widetilde q;G)$ at a second point.
\end{claim}

\begin{proof}[Proof of Claim \ref{claim:stable.unstable.loop}]
Let $C>0$ bound the displacement of $\widetilde H_g$ from $\iota$, as in \eqref{eq:hhu-lifted-semiconjugacy}. The tree action induced by $Q$ fixes $e$
and expands distances in $\widehat{\mathcal T}^s$ by $\lambda^k$. Iterating the fold of Claim~\ref{claim:4} therefore allows us to assume $d_s(e,t)>4C$.
\begin{figure}[htbp]
    \centering
    \begin{tikzpicture}[
  x=1.05cm,y=1.05cm,
  font=\small,
  >=Stealth,
  line cap=round,line join=round
]
  \definecolor{unstableblue}{HTML}{245DB0}
  \definecolor{stablegreen}{HTML}{087F72}
  \definecolor{intervalorange}{HTML}{D67818}
  \definecolor{leafgrey}{HTML}{747C87}
  \definecolor{diskfill}{HTML}{EAF0F7}
  \pgfmathsetmacro{\r}{sqrt(5.6)}

  \fill[diskfill]
    (-\r,2.8) -- (\r,2.8)
    -- plot[domain=\r:-\r,samples=101] (\x,{0.5*\x*\x})
    -- cycle;
  \node[font=\Large,text=leafgrey] at (-0.6,1.98) {$D$};

  \draw[leafgrey,<->,line width=0.8pt] (-4,2.8) -- (4,2.8);
  \draw[leafgrey,line width=1.4pt] (-\r,2.8) -- (\r,2.8);
  \node[leafgrey,anchor=south east] at (-2.87,2.94)
    {$\ell=q_s^{-1}(v)$};

  \draw[unstableblue,line width=0.8pt]
    plot[domain=-2.94:2.94,samples=101] (\x,{0.5*\x*\x});
  \draw[unstableblue,line width=1.8pt]
    plot[domain=-\r:\r,samples=101] (\x,{0.5*\x*\x});
  \fill[unstableblue] (-2.9,4.205) circle (1.7pt)
    node[above=4pt] {$a$};
  \fill[unstableblue] (2.9,4.205) circle (1.7pt)
    node[above=4pt] {$b$};
  \fill[leafgrey] (-\r,2.8) circle (1.5pt);
  \fill[leafgrey] (\r,2.8) circle (1.5pt);
  \node[unstableblue,anchor=east] at (-2.13,1.43)
    {$U\subset\widetilde W^u(\widetilde q;G)$};
  \draw[unstableblue,thin] (-2.06,1.51) -- (-1.84,1.693);

  \draw[intervalorange,line width=3.5pt]
    plot[domain=-0.73:0.66,samples=41] (\x,{0.5*\x*\x});
  \fill (0,0) circle (2pt);
  \node[anchor=north west] at (0.04,-0.17) {$\widetilde q$};
  \node[intervalorange,anchor=north] at (0.88,-0.44) {$I$};
  \draw[intervalorange,thin] (0.83,-0.36) -- (0.56,0.1568);

  \draw[stablegreen,line width=1.7pt]
    plot[domain=-1.35:3.92,samples=101]
      (\x,{-0.15*\x*\x+0.91*\x+0.468});
  \foreach \lo/\hi in {-0.83/-0.61,0.78/1.00,3.56/3.90}{
    \draw[stablegreen,line width=1.3pt,->]
      plot[domain=\lo:\hi,samples=12]
        (\x,{-0.15*\x*\x+0.91*\x+0.468});
  }
  \node[stablegreen,anchor=south] at (3.22,2.02)
    {$\widetilde W^s(z;G)$};

  \filldraw[fill=stablegreen,draw=white,line width=0.7pt]
    (-0.4,0.08) circle (2.3pt);
  \node[stablegreen,anchor=east] at (-1.32,-0.48)
    {$z\in I\cap Y_q$};
  \draw[stablegreen,thin] (-1.23,-0.43) -- (-0.49,-0.005);

  \filldraw[fill=stablegreen,draw=white,line width=0.7pt]
    (1.8,1.62) circle (2.3pt);
  \node[stablegreen,anchor=north] at (3.35,0.79)
    {second crossing};
  \draw[stablegreen,thin] (3.23,0.89) -- (1.88,1.57);
\end{tikzpicture}
    \caption{Claim \ref{claim:stable.unstable.loop}}
    \label{fig:claim.5}
\end{figure}

In the interior of the tree segment $[e,t]$, choose a point $v$ representing a nonsingular stable leaf, with
\[
  d_s(e,v)>C,\qquad d_s(v,t)>C.
\]
The interior of this segment lies in $\mathcal T^s$, and only countably many of its points represent singular leaves, so
such a choice is possible. Put $\ell=q_s^{-1}(v)$. The displacement bound gives
\[
  d_s(q_s(\widetilde q),e)\leq C,\qquad d_s(q_s(a),t)\leq C,\qquad d_s(q_s(b),t)\leq C.
\]
Thus $q_s(\widetilde q)$ lies on the $e$-side of $v$ in the tree, while $q_s(a)$ and $q_s(b)$ lie on the $t$-side. Every path from $\widetilde q$ to $a$ or $b$ must meet $\ell$, because its $q_s$-image must pass through $v$.

Let $U$ be the closure of the component of the unstable arc $[a,b]\setminus\ell$ containing $\widetilde q$. The arc $U$ and the segment of $\ell$ joining its endpoints bound a Jordan disk $D$. Choose a small interval $I\subset\operatorname{int}U$ about $\widetilde q$ such that
\[
  d_s\bigl((\widetilde H_g)_s(z),v\bigr)>C
  \qquad(z\in I).
\]
Continuity gives this inequality, and Claim~\ref{claim:periodic.saddle.related.to.p} gives $m_q^u(I)>0$.

A stable manifold through $I$ cannot meet $\ell$. Indeed, $(\widetilde H_g)_s$ is constant on that manifold, whereas the displacement bound places $(\widetilde H_g)_s(\ell)$ within distance $C$ of $v$.

For $z\in I\cap Y_q$, Claim~\ref{claim:Yq} gives a stable manifold transverse to $I$, with unbounded images on both sides. One half-leaf
enters $D$. It cannot stay there, since the image of the compact disk $D$ is bounded. It must leave through $U$, as it cannot meet $\ell$. This gives the second topological crossing. Since $Y_q$ is conull, the claim follows. See Figure \ref{fig:claim.5}.
\end{proof}
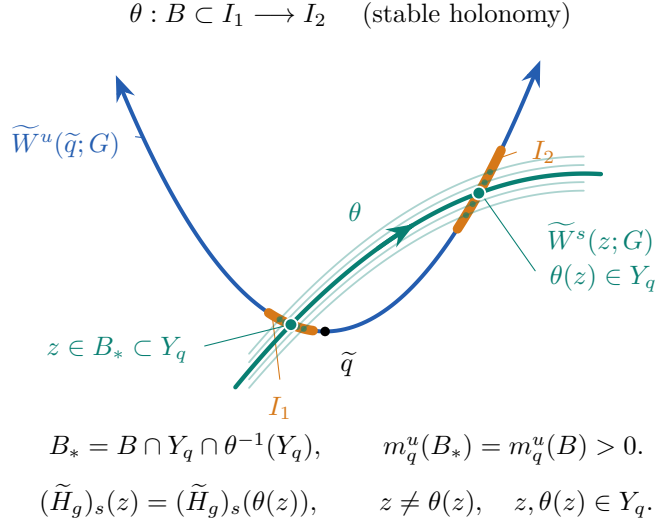
\begin{figure}[htbp]
    \centering
\begin{tikzpicture}[
  x=1.13cm,y=1.13cm,
  font=\small,
  >=Stealth,
  line cap=round,line join=round
]
  \definecolor{unstableblue}{HTML}{245DB0}
  \definecolor{stablegreen}{HTML}{087F72}
  \definecolor{intervalorange}{HTML}{D67818}
  \definecolor{annotationgrey}{HTML}{626B76}

  \draw[unstableblue,line width=1.5pt,<->]
    plot[domain=-2.45:2.52,samples=101]
      (\x,{0.5*\x*\x});
  \node[unstableblue,anchor=east] at (-2.28,2.22)
    {$\widetilde W^u(\widetilde q;G)$};
  \draw[unstableblue,thin] (-2.22,2.28) -- (-2.13,2.269);
  \fill (0,0) circle (1.8pt);
  \node[anchor=north west] at (0.06,-0.12) {$\widetilde q$};

  \draw[intervalorange,line width=3.6pt]
    plot[domain=-0.66:-0.14,samples=25]
      (\x,{0.5*\x*\x});
  \draw[intervalorange,line width=3.6pt]
    plot[domain=1.55:2.06,samples=25]
      (\x,{0.5*\x*\x});
  \node[intervalorange,anchor=north] at (-0.53,-0.64) {$I_1$};
  \draw[intervalorange,thin] (-0.53,-0.56) -- (-0.62,0.1922);
  \node[intervalorange,anchor=west] at (2.30,2.10) {$I_2$};
  \draw[intervalorange,thin] (2.24,2.10) -- (2.01,2.0201);

  \foreach \c in {-0.20,-0.10,0.10,0.20}{
    \draw[stablegreen,opacity=0.34,line width=0.7pt]
      plot[domain=-0.89:3.03,samples=80]
        (\x,{-0.15*\x*\x+0.91*\x+0.468+\c});
    \pgfmathsetmacro{\xl}{(0.91-sqrt(2.0449+2.6*\c))/1.3}
    \pgfmathsetmacro{\xr}{(0.91+sqrt(2.0449+2.6*\c))/1.3}
    \fill[stablegreen,opacity=0.60]
      (\xl,{0.5*\xl*\xl}) circle (1.15pt);
    \fill[stablegreen,opacity=0.60]
      (\xr,{0.5*\xr*\xr}) circle (1.15pt);
  }

  \draw[stablegreen,line width=1.65pt]
    plot[domain=-1.05:3.23,samples=90]
      (\x,{-0.15*\x*\x+0.91*\x+0.468});
  \draw[stablegreen,line width=1.65pt,->]
    plot[domain=0.62:1.05,samples=15]
      (\x,{-0.15*\x*\x+0.91*\x+0.468});
  \node[stablegreen,anchor=south east] at (0.55,1.17) {$\theta$};
  \node[stablegreen,anchor=north west] at (2.48,1.40)
    {$\widetilde W^s(z;G)$};

  \filldraw[fill=stablegreen,draw=white,line width=0.7pt]
    (-0.4,0.08) circle (2.4pt);
  \filldraw[fill=stablegreen,draw=white,line width=0.7pt]
    (1.8,1.62) circle (2.4pt);

  \node[stablegreen,anchor=east] at (-1.49,-0.21)
    {$z\in B_*\subset Y_q$};
  \draw[stablegreen,thin] (-1.41,-0.17) -- (-0.49,0.06);
  \node[stablegreen,anchor=west] at (2.51,0.63)
    {$\theta(z)\in Y_q$};
  \draw[stablegreen,thin] (2.46,0.71) -- (1.87,1.55);

  \node[anchor=south] at (0.3,3.43)
    {$\theta:B\subset I_1\longrightarrow I_2$
     \quad (stable holonomy)};
  \node[align=center] at (0.25,-1.37)
    {$B_*=B\cap Y_q\cap\theta^{-1}(Y_q),$
     \qquad $m_q^u(B_*)=m_q^u(B)>0.$};
  \node[align=center] at (0.25,-1.98)
    {$(\widetilde H_g)_s(z)=(\widetilde H_g)_s(\theta(z)),$
     \qquad $z\ne\theta(z),\quad z,\theta(z)\in Y_q.$};
\end{tikzpicture}

    \caption{Contradiction to injectivity of $(\widetilde H_g)_s$ on $Y_q$.}
    \label{fig:thm.entropy.bound}
\end{figure}

Claim~\ref{claim:stable.unstable.loop} supplies the two crossings; it remains to choose them in the set of injectivity. Let $E=I\cap Y_q$, which has positive $m_q^u$-measure. By Claim~\ref{claim:Yq}, the stable manifold through each $z\in E$ meets $\widetilde W^u(\widetilde q;G)$ transversely at $z$.
By Claim~\ref{claim:stable.unstable.loop}, it has a second topologically transverse intersection with this same unstable manifold.

The HHTU Sard argument, as used in \cite[Proposition~4.5]{CHHU2026}, shows that the exceptional subset of $E$ whose stable manifolds have a nontransverse second intersection has zero transverse Hausdorff dimension. Since $m_q^u$ gives zero mass to such sets, there is a subset
$B\subset E$ of positive $m_q^u$-measure for which the second intersections are transverse.

After restricting $B$ while retaining positive measure, these second intersections define a local stable holonomy
\[
  \theta:B\subset I_1\longrightarrow I_2,
\]
where $I_1$ and $I_2$ are disjoint subarcs of $\widetilde W^u(\widetilde q;G)$. Both intersections defining $\theta$ are transverse. Ben Ovadia's absolute continuity result \cite[Theorems~5.7 and~6.1]{BO2023} therefore implies that $\theta$ preserves the corresponding leaf measure
classes and, in particular,
\[
  m_q^u(\theta(B))>0.
\]
Since $Y_q$ is conull, we can therefore require both ends of a holonomy pair to belong to $Y_q$:
\[
  B_*:=B\cap Y_q\cap\theta^{-1}(Y_q),
  \qquad m_q^u(B\setminus B_*)=0.
\]
For $z\in B_*$, the distinct points $z$ and $\theta(z)$ lie on the same lifted stable manifold, so
\[
  (\widetilde H_g)_s(z)=(\widetilde H_g)_s(\theta(z)).
\]
Both points belong to $Y_q$, contradicting Claim~\ref{claim:Yq}. See Figure \ref{fig:thm.entropy.bound}.
\end{proof}

\begin{corollary}\label{cor:entropy.minimization}
If $g$ is a $C^\infty$ diffeomorphism isotopic to $f$, then
\[
  h_N(g)=h_{\mathrm{top}}(g) \quad\Longleftrightarrow\quad h_{\mathrm{top}}(g)=h_{\mathrm{top}}(f).
\]
In particular, whenever $h_{\mathrm{top}}(g)>h_{\mathrm{top}}(f)$,
\[
  h_{\mathrm{top}}(f) \leq h_N(g)=h_{\mathrm{top}}(\Phi_g)<h_{\mathrm{top}}(g).
\]
\end{corollary}

\begin{proof}
If $h_{\mathrm{top}}(g)=h_{\mathrm{top}}(f)$, the bounds
\[
  h_{\mathrm{top}}(f)\leq h_N(g)\leq h_{\mathrm{top}}(g)
\]
give equality. Conversely, suppose that $h_N(g)=h_{\mathrm{top}}(g)$. By \Cref{thm:entropy.formula},
\[
  h_{\mathrm{top}}(\Phi_g)=h_{\mathrm{top}}(g).
\]
If this common value exceeded $h_{\mathrm{top}}(f)$, \Cref{thm:strict-entropy-gap} would give a contradiction.
Hence $h_{\mathrm{top}}(g)=h_{\mathrm{top}}(f)$. The final assertion follows from \Cref{thm:entropy.formula} and the same entropy bounds.
\end{proof}

\subsection{Periodic-point growth and topological entropy}
\label{sec:kaloshin-bowen}

The entropy formula for Nielsen classes does not extend to counting individual periodic points. Indeed, every pseudo-Anosov isotopy
class contains a $C^\infty$ diffeomorphism $g$ for which $\operatorname{Fix}(g^n)$ is finite for every $n\ge1$, but
\[
  h_{\mathrm{top}}(g)<\limsup_{n\to\infty}\frac1n\log\max\{1,\#\operatorname{Fix}(g^n)\}=+\infty.
\]

To see this, first create a dissipative quadratic homoclinic tangency by a local modification of a smooth representative. This does not change its isotopy class. Kaloshin proved that, after a perturbation, there is an open set of diffeomorphisms $\mathcal U$ such that, for a residual subset $\mathcal R\subset U$, the number of isolated periodic points grows faster than any prescribed sequence. More precisely, if $P_n(g)$ denotes the number of isolated fixed points of $g^n$, then, for every sequence $(a_n)$ of positive
integers, the condition
\[
  \limsup_{n\to\infty}\frac{P_n(g)}{a_n}=+\infty
\]
holds for every $g \in \mathcal R$ \cite[Theorem~3]{Kaloshin2000}. This conclusion also holds in the $C^\infty$ topology
\cite[Introduction, comment~(i)]{Burguet2020}.

Apply this result with $a_n=\lceil e^{n^2}\rceil$ and choose $g$ in the intersection of this residual set with the
Kupka--Smale residual set. All fixed points of each iterate are then hyperbolic and hence isolated; compactness gives
\[
  \#\operatorname{Fix}(g^n)=P_n(g)<\infty.
\]
For infinitely many $n$, this number exceeds $e^{n^2}$, so its exponential growth rate is infinite. In contrast,
$h_{\mathrm{top}}(g)$ is finite because $g$ is Lipschitz on the compact surface. See \cite[Theorem 3.2.9]{KatokHasselblatt1995}.

\section{Continuity of the quotient entropy}
\label{sec:quotient-entropy-continuity}
\begin{proof}[Proof of \Cref{thm:quotient-entropy-continuity}]
\setcounter{claim}{0}
Let $g_j\to g$ in the $C^0$ topology. By \eqref{eq:prelim-quotient-continuity}, the maps $H_{g_j}$ converge uniformly to $H_g$, and $L(g_j)$ converges to $L(g)$ in the Hausdorff topology. As recalled in \Cref{subsec:prelim-continuity}, item~(i),
\[
  K=L(g)\cup\bigcup_{j\geq1}L(g_j)
\]
is a compact $\Phi$-invariant set on which $\Phi$ is expansive. We shall compare all the quotient systems inside this fixed set.

\begin{claim}\label{claim:continuity-upper}
One has
\[
  \limsup_{j\to\infty}\htop(\Phi_{g_j})\leq\htop(\Phi_g).
\]
\end{claim}

\begin{proof}[Proof of Claim \ref{claim:continuity-upper}]
Choose a finite open cover $\mathcal U$ of $K$ whose mesh is smaller than an expansivity constant of $\Phi|_K$. For a fixed $m\geq1$, let $\mathcal U^{(m)}$ denote its refinement along the first $m$ iterates, as in \Cref{subsec:prelim-continuity}, item~(ii). A minimal subfamily of $\mathcal U^{(m)}$ covering $L(g)$ also covers an open neighborhood of $L(g)$ in $K$. Hausdorff convergence therefore gives
\[
  N(\mathcal U^{(m)}|_{L(g_j)})\leq N(\mathcal U^{(m)}|_{L(g)})
\]
for all sufficiently large $j$. Applying \eqref{eq:prelim-cover-entropy}, we obtain
\[
  \limsup_{j\to\infty}\htop(\Phi_{g_j})\leq\frac1m\log N(\mathcal U^{(m)}|_{L(g)}).
\]
Taking the infimum over $m$ and using the same formula for $L(g)$ proves the claim.
\end{proof}

Assume now that $g$ is $C^{1+\alpha}$, and fix $0<a<\htop(\Phi_g)$. To obtain the reverse inequality, we will find a horseshoe whose image in the quotient has entropy larger than $a$ and remains present under perturbation.

\begin{claim}\label{claim:continuity-horseshoe}
There are $n\geq1$ and a locally maximal hyperbolic horseshoe $\Lambda$ for $g^n$ such that $H_g|_\Lambda$ is injective and
\[
  \frac1n\htop(g^n|_\Lambda)>a.
\]
\end{claim}

\begin{proof}[Proof of Claim \ref{claim:continuity-horseshoe}]
By the variational principle we can choose an ergodic $\Phi_g$-invariant probability measure $\nu$ with $h_\nu(\Phi_g)>a$. Let $\mu$ be an ergodic invariant lift under $H_g$. 

Choose $\epsilon>0$ so that $h_\nu(\Phi_g)-2\epsilon>a$. The Brin--Katok estimate \eqref{eq:prelim-uniform-bowen-bound} supplies a radius $r>0$
and a measurable set $A\subset L(g)$ of positive $\nu$-measure on which the Bowen-ball estimate holds uniformly for all
sufficiently large $n$. We shall choose our return points in $H_g^{-1}(A)$, so that the same estimate applies to
each of their images in the quotient.

Choose $\eta>0$ so small that
\[
  d_M(y,z)<\eta\quad\Longrightarrow\quad d(H_g(y),H_g(z))<r/3.
\]
Apply the return-rectangle construction in \Cref{subsec:prelim-continuity}, item~(iv), to the measure
$\mu$, this value of $\eta$, and the positive-measure set $H_g^{-1}(A)$. It provides a rectangle $R$ and a measurable set $B\subset H_g^{-1}(A)\cap R$ with $\mu(B)>0$.

Applying \eqref{lem:prelim-separated-returns} to $B$, we get arbitrarily large $n$ and finite sets $E_n\subset B\cap g^{-n}B$ such that $H_g(E_n)$ is $(n,r)$-separated, $H_g$ is injective on $E_n$, and
\[
  \#E_n\geq c_Be^{n(h_\nu(\Phi_g)-\epsilon)}.
\]
Choose such an $n$ large enough for the return-rectangle construction and for
\[
  \frac1n\log\#E_n>a.
\]

For each $x\in E_n$, let $V_x$ be the corresponding vertical return strip. By \eqref{eq:prelim-return-accuracy} and the
choice of $\eta$, every $y\in V_x$ satisfies 
\begin{equation}\label{eq:thm51-quotient-accuracy}
  d_n^{\Phi_g}(H_g(y),H_g(x))<r/3,
\end{equation}
where $d_n^{\Phi_g}$ is the Bowen distance defined in \Cref{subsec:prelim-continuity}, item~(iii). These strips are pairwise disjoint: a point in
$V_x\cap V_{x'}$, with $x\ne x'$, would place $H_g(x)$ and $H_g(x')$ at Bowen distance less than $2r/3$, contradicting
their separation.

The Katok--Mendoza construction \eqref{eq:prelim-return-horseshoe} therefore gives a horseshoe $\Lambda$ for $g^n$, whose points are uniquely determined
by their bi-infinite itineraries among the strips $V_x$. We check that $H_g$ distinguishes these points as well. If $y,z\in\Lambda$ are distinct, their itineraries differ at some $k\in\mathbb Z$. Thus $g^{kn}y\in V_x$ and $g^{kn}z\in V_{x'}$ for distinct $x,x'\in E_n$.
The separation of $H_g(E_n)$ and \eqref{eq:thm51-quotient-accuracy} imply
\[
  d_n^{\Phi_g}\bigl(H_g(g^{kn}y),H_g(g^{kn}z)\bigr)>r/3.
\]
If $H_g(y)=H_g(z)$, the semiconjugacy identity would make this distance zero. Hence $H_g|_\Lambda$ is injective.
Finally, \eqref{eq:prelim-horseshoe-entropy} gives 
\[
  \frac1n\htop(g^n|_\Lambda)
  =\frac1n\log\#E_n>a.
\]
\end{proof}

\begin{claim}\label{claim:continuity-persistent-image}
For all sufficiently large $j$, there is a compact $g_j^n$-invariant set $\Lambda_j$ such that
\[
  H_{g_j}(\Lambda_j)=H_g(\Lambda).
\]
\end{claim}

\begin{proof}[Proof of Claim \ref{claim:continuity-persistent-image}]
Since $g_j^n\to g^n$ uniformly, the semistability statement \eqref{eq:prelim-horseshoe-semistability} provides compact $g_j^n$-invariant sets $\Lambda_j$ and continuous surjections $r_j\colon\Lambda_j\to\Lambda$ satisfying
\[
  r_j\circ g_j^n=g^n\circ r_j,
  \qquad
  \sup_{x\in\Lambda_j}d_M(x,r_j(x))\longrightarrow0.
\]
Uniform convergence of $H_{g_j}$ and uniform continuity of $H_g$ give
\[
  \sup_{x\in\Lambda_j}d\bigl(H_{g_j}(x),H_g(r_j(x))\bigr)\longrightarrow0.
\]
Thus $H_{g_j}|_{\Lambda_j}$ and $H_g\circ r_j$ are uniformly close maps into $K$. Both intertwine $g_j^n$ with $\Phi^n$: for the first map this is the quotient semiconjugacy identity, and for the second it follows from the identity for $r_j$. The uniqueness consequence of expansivity recalled in
\Cref{subsec:prelim-continuity}, item~(ii), therefore gives 
\[
  H_{g_j}|_{\Lambda_j}=H_g\circ r_j
\]
for all sufficiently large $j$. Since $r_j$ is onto, their images equal $H_g(\Lambda)$, proving the claim.
\end{proof}

The continuous injection $H_g|_\Lambda$ is a homeomorphism onto its image, and it conjugates $g^n|_\Lambda$ to
$\Phi^n|_{H_g(\Lambda)}$. By Claim \ref{claim:continuity-persistent-image}, this same compact $\Phi^n$-invariant set is contained in $L(g_j)$ for all
sufficiently large $j$. Since entropy scales by $n$ under
iteration,
\[
\begin{aligned}
  \htop(\Phi_{g_j})&=\frac1n\htop(\Phi^n|_{L(g_j)})\\
  &\geq\frac1n\htop(\Phi^n|_{H_g(\Lambda)})=\frac1n\htop(g^n|_\Lambda)>a.
\end{aligned}
\]
Letting $a\uparrow\htop(\Phi_g)$ proves lower semicontinuity at $g$. Together with Claim \ref{claim:continuity-upper}, this
proves the theorem.
\end{proof}

\section{Questions}\label{sec:questions}
The following questions remain open:
\begin{question}
  Do Theorems \ref{thm:entropy.formula} and \ref{thm:strict-entropy-gap} hold for arbitrary homeomorphisms in the isotopy class of a pseudo-Anosov map? Is $g\mapsto h_{\mathrm{top}}(\Phi_g)$ continuous throughout this isotopy class?
\end{question}

\begin{question}[Periodic growth for analytic maps]
Does every pseudo-Anosov isotopy class contain a real-analytic diffeomorphism \(g\), with \(\#\operatorname{Fix}(g^n)<\infty\) for every \(n\), whose exponential periodic-point growth rate exceeds \(h_{\mathrm{top}}(g)\)? Namely, 
\[
\htop(g)<\limsup_{n\to\infty}\frac1n\log\max\{1,\#\operatorname{Fix}(g^n)\}?  
\]
\end{question}

\begin{question}[Strict monotonicity]
Does a strictly larger quotient necessarily have greater entropy? Namely, does 
\[
L(g_1)\subsetneq L(g_2)\quad\Rightarrow\quad \htop(\Phi_{g_1})<\htop(\Phi_{g_2})?
\]
This question was already stated in \cite[Question 4.1]{HHU2004}. After \Cref{thm:entropy.formula}, it asks whether additional global orbit classes necessarily increase the exponential growth rate of their Nielsen classes.

The case where $g_1=f$ follows from \cite{Handel1988} and the entropy formula found in \Cref{thm:entropy.formula}.
\end{question}

\begin{question}[Entropy gap]This question refines the minimization problem posed in \cite[Question~4.5]{HHU2004}, which asks whether each class of
representatives with the same quotient admits a representative of minimal entropy.

Let $K=L(g_0)$ for some $C^\infty$ diffeomorphism $g_0\simeq f$, and suppose that
\[
  h_{\mathrm{top}}(\Phi|_K)>h_{\mathrm{top}}(f).
\]
Does there exist $\delta(K)>0$ such that
\[
  h_{\mathrm{top}}(g)\ge h_{\mathrm{top}}(\Phi|_K)+\delta(K)
\]
for every $C^\infty$ diffeomorphism $g\simeq f$ satisfying $L(g)=K$?\par

Our theorem gives a positive entropy gap for each smooth representative. The question is whether this gap can tend to zero while the quotient remains fixed. A negative answer would yield a sequence of $C^\infty$ diffeomorphisms $g_j\simeq f$ with $L(g_j)=K$ and
\[
  h_{\mathrm{top}}(g_j)\longrightarrow h_{\mathrm{top}}(\Phi|_K).
\]
In that case, the quotient entropy would be the unattained infimum of the entropies of its smooth realizations.

Does this answer change if we allow $g_0$ and $g$ to be arbitrary homeomorphisms?
\end{question}

\begin{question}
    Does the specification property hold for all $\Phi_g$?
\end{question}

\section*{Acknowledgments}

The authors acknowledge the assistance of ChatGPT (OpenAI) in discussing and refining arguments, checking bibliographical references, improving the exposition, and preparing LaTeX text and figures. The authors take full responsibility for the mathematical content and the final manuscript.

\bibliographystyle{amsplain}
\bibliography{referencecorrected}

\makeatletter
\@setaddresses
\global\let\@setaddresses\relax
\makeatother

\clearpage
\pagestyle{plain}
\appendix

\begingroup
\scriptsize
\renewcommand{\theenumi}{\thesubsection.\arabic{enumi}}
This appendix reproduces the manuscript “Dynamics in the isotopy class of a pseudo-Anosov map” by Federico Rodriguez-Hertz, Jana Rodriguez-Hertz, and Raúl Ures, originally posted on arXiv in 2004 (\cite{HHU2004}) but never published in a journal. We include it here for completeness and ease of reference, with only typographical and formatting adjustments.
\section{Dynamics in the isotopy class of a pseudo-Anosov map}
\label{rrhu:appendix}
\begin{center}
Federico Rodriguez-Hertz, Jana Rodriguez-Hertz, and Ra\'{u}l Ures
\end{center}
\noindent Despite its homotopical stability, new relevant dynamics appear in the
isotopy class of a pseudo-Anosov homeomorphism. We study these new
dynamics by identifying homotopically equivalent orbits, obtaining
a more complete description of the topology of the corresponding quotient
spaces, and their stable and unstable sets. In particular, we get
some insight on how new periodic points appear, among other
corollaries. A list of further questions and problems is added at
the end of the paper.

\subsection{Introduction}
In this paper we consider a pseudo-Anosov homeomorphism $f:M^2\to
M^2$ on a closed surface, and homeomorphisms $g:M^2\to M^2$ in its
isotopy class. The reader not familiar with these concepts may consult \cite{rrhu-flp}.
Following \cite{rrhu-h}, we shall use Katok's {\sffamily\itshape global shadowing}
to compare orbits of $f$ with orbits of $g$. According to this
definition the $f$-orbit of $x$ is globally shadowed by the
$g$-orbit of $y$ if $\sup_{n\in\mathbb Z}D(F^n(\widetilde x), G^n(\widetilde y))<\infty$,
for some adequate lifts $F,G,\widetilde x,\widetilde y$ of $f,g,x,y$, respectively,
to the universal cover ($D$ is any equivariant metric). For simplicity we denote $(f,x)\sim(g,y)$.\par
This equivalence relation permitted to state a kind of homotopical
stability of $f$ in its isotopy class, in the sense that for all $g\simeq f$
each $f$-orbit $(f,x)$ is globally shadowed by (at least) one $g$-orbit $(g,y_x)$
\cite{rrhu-l1,rrhu-h,rrhu-f}. That is, the dynamics of $f$ is repeated all over
its isotopy class. However, unlike the Anosov case, new dynamics
may appear for $g$, which are not globally shadowed by $f$.\par
In this work we study these new dynamics using the following
approach. Given a fixed $g$, we identify all $(g,x)\sim(g,x')$,
obtaining a quotient space $M_L(g)$. The map $g$ naturally induces a homeomorphism $g_L$ on $M_L(g)$.
Observe that, in general $M_L(g)$ is no longer a surface, though $g_L$ is still expansive.\par
 The core of this paper is based on an observation,
due to Fathi \cite{rrhu-f}, of the existence of a ``universal torus" ${\mathcal Z}$ and a hyperbolic extension $\Phi$ of $f$ in ${\mathcal Z}$, for
which $M$ is the smallest non trivial invariant sub-continuum.
The universal cover of ${\mathcal Z}$ is the product of
two trees obtained from the stable and unstable foliations,
respectively.\par
Our first theorem states that all $M_L(g)$ are naturally embedded
in ${\mathcal Z}$, and that $g_L$ is nothing but $\Phi$ restricted to $M_L(g)$
\begin{rrhutheorem}\label{rrhu:teo1}
Let $f:M^2\to M^2$ be a pseudo-Anosov homeomorphism on a closed
surface. Then there exists ${\mathcal Z}$ and a homeomorphism $\Phi:{\mathcal Z}\to{\mathcal
Z}$, such that:
\begin{enumerate}
\item\label{rrhu:teo1.1} $\Phi:{\mathcal Z}\to{\mathcal Z}$ is a {\sffamily\itshape metrically split
 hyperbolic homeomorphism}.
\item \label{rrhu:teo1.2}$\Phi$ has the pseudo-orbit tracing property on
${\mathcal{Z}}$. Additionally, the pseudo orbits of any lift of
$\Phi$ are uniquely shadowed in the universal cover of ${\mathcal{Z}}$.
\item \label{rrhu:teo1.3}For all $g:M^2\to M^2$ homotopic to $f$, there exists a
$\Phi$-invariant subset $L(g)\subset{\mathcal{Z}}$ on which $g_L$ is
conjugate to $\Phi|_{L(g)}$
\end{enumerate}
\end{rrhutheorem}
\subsubsection*{Remark}
A {\sffamily\itshape hyperbolic homeomorphism} is to be understood in the sense of Bowen, that is, one having hyperbolic
canonical coordinates \cite{rrhu-b2}. The definition of {\sffamily\itshape metrically splitting} is, {\itshape mutatis
mutandis}, as in Franks (\cite{rrhu-fr}, p.67).

Theorem \ref{rrhu:teo1} provides a unified scenario for all $M_L(g)$ and $g_L$, endowed with a hyperbolic
structure. From now on, we shall write $L(g)$ and $\Phi_g=\Phi|_{L(g)}$ instead of $M_L(g)$ and $g_L$.
We find that:
\begin{rrhutheorem}\label{rrhu:cont.coc} The application $g\mapsto L(g)\in 2^{\mathcal{Z}}$ is
continuous in the isotopy class of $f$, $2^{\mathcal{Z}}$ endowed with the
Hausdorff topology. The set $\{L(g):g\simeq f\}\subset
2^{\mathcal{Z}}$ is arcwise connected
\end{rrhutheorem}
As a Corollary we obtain
\begin{rrhucorollary}\label{rrhu:top.est}
$\Phi_g$ is never topologically stable. If $g$ and $g'$ are
semiconjugate by a semiconjugacy homotopic to the identity then
$L(g)=L(g')$.
\end{rrhucorollary}
\par
Corollary \ref{rrhu:top.est} shows that each $\Phi_g$ has a distinct
dynamics.
 Let us point out that always when a pair $g$, $g'$ are
 semiconjugate, a semiconjugacy can be chosen to be homotopic to a
 symmetry of $f$ (see comments in \S \ref{rrhu:sec.preg}). Although previous statements give some light to the new dynamics
appearing in the isotopy class of $f$, they mainly open many
questions and problems, in \S\ref{rrhu:sec.preg} we state some of
them.\par
It is expectable that the study, and specially the local
analysis of the sets $L(g)$ will bring information about the
dynamics of the $\Phi_g$, and are an interesting topic in its own.
Following result describes the local properties of $L(g)$, new
periodic points, and local stable and unstable sets.

 \begin{rrhutheorem}\label{rrhu:teo.local}Let $g:M^2\to M^2$ be a homeomorphism on
 a closed surface, $g\simeq f$. Then:
 \begin{enumerate}
 \item \label{rrhu:teo4.1} The topological dimension of $L(g)$ is $2$.
 \item \label{rrhu:teo4.2} There exists $\varepsilon>0$ such that $CW^s_{\varepsilon}(x)$
 and $CW^u_\varepsilon(x)$ are nontrivial dendrites for all $x\in L(g)$,
 with diameters bounded from below.
 \item \label{rrhu:teo4.3}If $p\in{\mathcal P}er(\Phi_g)\setminus{\mathcal
 P}er(\Phi_f)$ then $p$ is an endpoint of {\sffamily\itshape both}
 $CW^s_{\varepsilon}(p)$ and $CW^u_\varepsilon(p)$.
 \end{enumerate}
 \end{rrhutheorem}
The {\sffamily\itshape $\varepsilon$-stable component of $x$}, $CW^s_\varepsilon (x)$, denotes the
component of $x$ in $W^s_\varepsilon(x)$, the $\varepsilon$-local stable set of
$x$. A dendrite is a uniquely arcwise connected curve.\par
 Item \ref{rrhu:teo4.2} improves \cite{rrhu-cs} where, using delicate arguments of plane topology,
it was shown arcwise connectedness of $CW^\sigma_\varepsilon(x)$ for $g$ $C^0$-close to
an expansive homeomorphism. It also extends the result over the isotopy class of
$f$.

\subsection{Global Properties}\label{rrhu:seccion.global}
From now on, $f:M\rightarrow M$ will be a pseudo-Anosov homeomorphism and
$F:\widetilde {M}\rightarrow \widetilde {M}$ any lift of $f$ to the universal covering of $M$.
 Associated to $F$ there are two transversal invariant foliations with singularities,
 $\mathcal F^{u}$ and $\mathcal F^{s}$, with transverse measures contracted and expanded by $F$
 with rates $\lambda>1$ and $\lambda^{-1}$, respectively.\par
We briefly describe some remarkable properties of the construction used by A. Fathi to embed
a pseudo-Anosov homeomorphism in a hyperbolic dynamics. The reader
interested in the details should consult \cite{rrhu-f}.
\begin{enumerate}
\item \label{rrhu:item.1} The set of $\sigma$-leaves ${\mathcal{T}}^\sigma=\widetilde M/_{\widetilde {\mathcal{F}}^\sigma}$ is a
non locally compact, not complete tree, when the distance of two points in ${\mathcal{T}}^\sigma$ is defined as the
minimum over the transverse measures of arcs joining their associated
leaves ($\sigma=s,u$)
\cite{rrhu-ms}.\vspace*{0.5em}
\item The completion $\widehat{{\mathcal{T}}}^\sigma$ of ${\mathcal{T}}^\sigma$ is still a non
locally compact tree \cite{rrhu-ms}, on which $F$ and the fundamental
group $\Gamma$ of $M$ naturally act, since the covering
transformations of $F$ are isometries of the metrics defined in
(\ref{rrhu:item.1}).\vspace*{0.5em}
\item The metric associated to $\widehat{{\mathcal{T}}}^s$ is $\lambda$-expanded
under the action of $F$, while the metric of $\widehat{{\mathcal{T}}}^u$ is
$\lambda^{-1}$-contracted.\vspace*{0.5em}
\item The product action $\widetilde \Phi$ on $\widetilde {\mathcal{Z}}=\widehat{{\mathcal{T}}}^s\times\widehat{{\mathcal{T}}}^u$ is a {\sffamily\itshape metrically split hyperbolic
homeomorphism}. ${\mathcal{Z}}=\widetilde {\mathcal{Z}}/_{\Gamma}$ is a metric space having
$\widetilde {\mathcal{Z}}$ as universal cover (the product metric is
$\Gamma$-equivariant). $\widetilde \Phi$ trivially induces a homeomorphism $\Phi$ on ${\mathcal{Z}}$.\vspace*{0.5em}
\item \label{rrhu:item2.5} The set $\widetilde {L(f)}=\{(l^s,l^u)\in\widetilde {\mathcal{Z}}:l^s\cap
l^u\ne\emptyset\}$ is homeomorphic to $\widetilde M$. Its projection to ${\mathcal{Z}}$, $L(f)$, is
$\Phi$-invariant and $\Phi_f=\Phi|_{L(f)}$ is conjugate to $f$.
Moreover, $L(f)$ is contained in all non trivial $\Phi$-invariant closed
connected subsets of ${\mathcal{Z}}$ \cite{rrhu-f}.
\end{enumerate}

Bowen's proof of the shadowing lemma \cite{rrhu-b}
works also in this setting, both for $\Phi$ and for $\widetilde \Phi$, since ${\mathcal{Z}}$ and $\widetilde {\mathcal{Z}}$ are complete.
Being $\widetilde {\mathcal{Z}}$ metrically split, the fact that $\widetilde \Phi$ is a
hyperbolic homeomorphism easily implies that every
$\widetilde \Phi$-pseudo-orbit is uniquely shadowed.\par
Let us prove that each $g_L$ is a restriction of $\Phi$ to a set homeomorphic to $M_L(g)$:

In view of \ref{rrhu:item2.5}, $\widetilde {L(f)}$ is a universal cover of $M$
and $\widetilde \Phi$ is a lift of $f$ when restricted to this set. Take
 a lift $G$ of $g$, equivariantly homotopic to
 $\widetilde \Phi|_{\widetilde {L(f)}}$. Then each $G$-orbit is a pseudo
 orbit of $\widetilde \Phi$. Shadowing Lemma provides a continuous map $\widetilde H:\widetilde {L(f)}\to \widetilde {\mathcal{Z}}$
that is a $\Gamma$-invariant semiconjugacy between $G$ and
$\widetilde \Phi|_{Im(\widetilde H)}$. This induces a semiconjugacy
$H:L(f)\to Im(H)\subset{\mathcal{Z}}$, with the property that
$x'\in H^{-1}(x)$ if and only if $(g,x)\sim(g, x')$. Hence
$Im(H)=L(g)$ is homeomorphic to $M_L(g)$ and $g_L$ is
conjugate to $\Phi|_{L(g)}$, ending proof of Theorem
\ref{rrhu:teo1}.\par

It is easy to prove that semiconjugacies that send $M$ into ${\mathcal{Z}}$
vary continuously with respect to homeomorphisms $g$ of $M$, in
the $C^0$ topology. This implies that $\mathcal L=\{L(g)\subset {\mathcal{Z}};\:
g\,\,\mbox{homeomorphism of}\,\,M\}$ is an arcwise-connected set
in the Hausdorff topology. In particular, no $L(g)$ is an isolated set of
$\Phi$. This implies they have not the shadowing property, which
in this context means they are not topologically stable.\par
For the proof of Corollary \ref{rrhu:top.est}, take a continuous $h:M\to M$ homotopic
to the identity such that $h\circ g=g'\circ h$. Consider a lift
$\widetilde h$ of $h$ to $\widetilde {L(f)}$, and the $\Gamma$-invariant
semiconjugacies $\widetilde H, \widetilde H':\widetilde {L(f)}\to \widetilde {\mathcal{Z}}$ obtained in the proof
of Theorem \ref{rrhu:teo1.3} for lifts $G$, $G'$ of $g$, $g'$
respectively. We claim that $\widetilde H=\widetilde H'\circ \widetilde h$ for a suitable $\widetilde h$. Indeed,
as $\widetilde h$ can be chosen to be at a bounded distance of the identity, we easily
obtain that the distance between $\widetilde \Phi^n\circ\widetilde H(x)(=\widetilde H\circ G^n(x))$
and $\widetilde \Phi^n\circ \widetilde H'\circ \widetilde h(x)(=\widetilde H'\circ \widetilde h\circ G^n(x)))$ is
uniformly bounded over $n\in{\mathbb Z}$. We conclude $\widetilde H(x)=\widetilde H'\circ \widetilde h(x)$
due to hyperbolicity of $\widetilde \Phi$. The fact that $\widetilde h$ is onto
implies $Im(H)=Im(H')$ and the Corollary follows.

\subsection{Local Properties}

This section is devoted to proving Theorem \ref{rrhu:teo.local}:
\noindent\begin{enumerate}

\item Item \ref{rrhu:teo4.1} is in part a consequence of Theorem
\ref{rrhu:teo1.3}, since ${\mathcal{Z}}$, being the product of two
trees, has topological dimension 2. On the other hand, $L(g)$
contains a copy of $M$ (item (\ref{rrhu:item2.5}) of Section
\S\ref{rrhu:seccion.global}, see also \cite{rrhu-l1,rrhu-h}), hence its topological dimension is 2.

\item Item \ref{rrhu:teo4.2} follows trivially from the fact that $CW^s_{\varepsilon}(x)$ are compact connected
subsets of a tree. These sets have uniform size due to \cite{rrhu-l1} ($L(g)$ are locally connected).

\item Let $p$ be a periodic point of $\Phi$. Take $\widetilde p=(l^s,l^u)$ a lift of $p$.  If $l^s\in {\mathcal{T}}^s$,
there exist a lift $\widetilde \Phi$ of $\Phi$ and a covering transformation $\tau$ such
that $\tau \widetilde \Phi^n(\widetilde p)=\widetilde p$, implying $l_s$ is periodic for $f:M\rightarrow M$. It follows that $l^s$
contains an n-periodic $q$ of $f$. There exists a lift $\widetilde q$ of $q$ to $\widetilde {L(f)}$
that is fixed by $\tau\widetilde  \Phi^n$. Since $\tau \widetilde \Phi^n_{{\mathcal{Z}}}$ has a unique fixed point we obtain that $p=q\in
M$. This shows a $\Phi$-periodic point is in $L(f)$ if and only if its corresponding leaves $l^s_p$ {\sffamily\itshape and} $l^u_p$ are,
respectively, in the trees ${\mathcal{T}}^s$ and ${\mathcal{T}}^u$. \par
Now, points in the completion $\hat{\mathcal{T}}^\sigma$ that are not in ${\mathcal{T}}^\sigma$, are endpoints of the tree
\cite{rrhu-f}, so we arrive to the desired conclusion.

\end{enumerate}
\subsection{Questions \& Comments}\label{rrhu:sec.preg}
The first set of questions concerns the relation between
topological entropy, periodic points, and the sets $L(g)$
\begin{rrhuquestion}
Is the application $g\mapsto h_{top}(\Phi_g)$ continuous? Is it
strictly increasing with respect to inclusions of $L(g)$'s?
\end{rrhuquestion}
\begin{rrhuquestion}
Does Bowen's formula apply for $h_{top}(\Phi_g)$?
\end{rrhuquestion}
\begin{rrhuquestion}\label{rrhu:p.per}
Are periodic points dense in $L(g)$?
\end{rrhuquestion}
Let us point out that, in case $g$ is Axiom A with the strong transversality
condition, it is known from \cite{rrhu-lu} that question \ref{rrhu:p.per} has an affirmative
answer.\par
Observe that $L(g)=L(g')$ does not necessarily imply $g$ is
semiconjugate to $g'$. However; an equivalence relation can be
defined for $g$ in the isotopy class of $f$, namely: $g\approx g'$
iff $L(g)=L(g')$. Is this relation dynamically relevant in some
sense? For example:
\begin{rrhuquestion}
Is every $L(g)$ realized by some Axiom A? That is, does every
$\approx$-class contain an Axiom A?
\end{rrhuquestion}
\begin{rrhuquestion}
Is there a minimal representative (e.g. in the sense of its number of periodic points, topological entropy, etc) in each
$\approx$-class?
\end{rrhuquestion}
\begin{rrhuquestion} Does any $\Phi$-invariant subcontinuum of ${\mathcal{Z}}$
correspond to a $L(g)$, for some $g$ in the isotopy class of $f$?
\end{rrhuquestion}
We would like also to observe that this question is still
unsolved:
\begin{rrhuquestion}
Is $\dim_p L(g)=2$ at {\em all} points $p\in L(g)$? Does $L(g)$
have a special topological structure, such as Cantor manifold,
etc.?
\end{rrhuquestion}
Finally, though it is not in the scope of our paper, it is
interesting to note that the existence of two semiconjugate
homeomorphisms in the isotopy class of a pseudo-Anosov map $f$, whose
semiconjugacy cannot be chosen homotopic to the identity indicates
the presence of symmetries of $f$.

\begingroup
\renewcommand{\refname}{References for Appendix \thesection}
\renewcommand{\bibliofont}{\scriptsize}

\makeatletter
\renewcommand{\@bibtitlestyle}{\section*{\refname}}
\makeatother
\hypersetup{bookmarksdepth=-2}

\endgroup

\endgroup

\end{document}